\documentclass[12pt]{amsart}

\usepackage{answers}
\usepackage{setspace}
\usepackage{enumitem}
\usepackage{multicol}
\usepackage{mathrsfs}
\usepackage[margin=1in]{geometry} 
\usepackage{amsmath,amsthm,amssymb}
\usepackage{listings}
\usepackage{xspace}
\usepackage{bbm}

\usepackage{amsfonts,amssymb,amsbsy}
\usepackage{latexsym}
\usepackage{color}
\usepackage{graphics} 
\usepackage{float}
\usepackage{subfigure}
\usepackage{epsfig} 
\usepackage{epstopdf}
\usepackage{caption}
\usepackage{longtable}
\usepackage{afterpage}
\usepackage{array,booktabs,enumitem}
\usepackage{algorithm}  
\usepackage{algorithmicx}  
\usepackage{algpseudocode}  
\usepackage{mathrsfs}
\usepackage{bm}
\usepackage{tikz}
\usepackage{hyperref}
\usepackage{cleveref}
\usepackage{multirow}
\newtheorem{theorem}{Theorem}[section]
\newtheorem{proposition}[theorem]{Proposition}
\newtheorem{lemma}[theorem]{Lemma}
\newtheorem{corollary}[theorem]{Corollary}
\theoremstyle{definition}
\newtheorem{definition}[theorem]{Definition}
\theoremstyle{remark}
\newtheorem{remark}[theorem]{Remark}

\newtheoremstyle{assumption}
{6pt}
{6pt}
{\rm}
{}
{\bfseries}
{}
{1pt}
{}
\theoremstyle{assumption}
\newtheorem{assumption}{Assumption} 

\usepackage[colorinlistoftodos,bordercolor=orange,backgroundcolor=orange!20,linecolor=orange,textsize=scriptsize]{todonotes}
\usepackage{frenchineq}
\newcommand{\veryshortarrow}[1][3pt]{\mathrel{%
		\hbox{\rule[\dimexpr\fontdimen22\textfont2-.2pt\relax]{#1}{.4pt}}%
		\mkern-4mu\hbox{\usefont{U}{lasy}{m}{n}\symbol{41}}}}

\makeatletter

\setbox0\hbox{$\xdef\scriptratio{\strip@pt\dimexpr
		\numexpr(\sf@size*65536)/\f@size sp}$}

\newcommand{\scriptveryshortarrow}[1][3pt]{{%
		\hbox{\rule[\scriptratio\dimexpr\fontdimen22\textfont2-.2pt\relax]
			{\scriptratio\dimexpr#1\relax}{\scriptratio\dimexpr.4pt\relax}}%
		\mkern-4mu\hbox{\let\f@size\sf@size\usefont{U}{lasy}{m}{n}\symbol{41}}}}

\makeatother

\newcommand{\fromsource}{_{\mathrm{s}\veryshortarrow}}

\newcommand{\sourceset}{\mathcal{K}_{\textrm{src}}}
\newcommand{\destset}{\mathcal{K}_{\textrm{dst}}}

\newcommand{\Far}{\text{\sc Far}\xspace}
\newcommand{\Accepted}{\text{\sc Accepted}\xspace }
\newcommand{\Narrow}{\text{\sc NarrowBand}\xspace } 

\newcommand{\Start}{\text{\sc Start}\xspace }

\newcommand{\lowerboundf}{\underline{f}}
\newcommand{\upperboundf}{\overline{f}}

\newcommand{\N}{\mathbb{N}}

\newcommand{\R}{\mathbb{R}}
\newcommand{\A}{{\mathcal A}}

\newcommand{\K}{{\mathcal K}}
\newcommand{\Cc}{{\mathcal C}}
\newcommand{\Pro}{\mathbb{P}}
\newcommand{\E}{\mathrm{E}}
\newcommand{\Var}{\mathrm{Var}}
\newcommand{\tp}{\mathrm{t}^+}

\newcommand{\Tra}{\mathrm{Tr}}

\newcommand{\Oc}{\mathcal{O}}
\newcommand{\Ocb}{\overline{\mathcal{O}}}

\newcommand{\dd}{\tilde{d}}

\newcommand{\subv}{\overline{v}}
\newcommand{\supv}{\underline{v}}

\newenvironment{skproof}{{\noindent\it Sketch of Proof.}}{\hfill $\square$\par}



\usepackage{geometry}

\renewcommand{\theenumi}{\roman{enumi}}%

\title[Convergence and Error Estimates of A Semi-lagrangian scheme]{Convergence and Error Estimates of A Semi-Lagrangian scheme for the Minimum Time Problem
}

\date{\today} 

\author{Marianne Akian}
\address[Marianne Akian]
{Inria and CMAP, \'Ecole polytechnique, IP Paris, CNRS}
\email{Marianne.Akian@inria.fr}
\author{Shanqing Liu$^{\, 1}$}
\address[Shanqing Liu]{Division of Applied Mathematics, Brown University}
\email{shanqing\_liu@brown.edu} 
\thanks{$(1)$ This work began when the second author was a PhD student at Inria and CMAP, Ecole polytechnique, CNRS, Institut Polytechnique de Paris}

\begin{document}

\begin{abstract}
	We consider a semi-Lagrangian scheme for solving the minimum time problem, with a given target, and the associated eikonal type  equation. 
	We first use a discrete time deterministic optimal control problem interpretation of the time discretization scheme, and show that the discrete time value function is semiconcave under regularity assumptions on the dynamics and the boundary of target set. We establish a convergence rate of order $1$ in terms of time step 
	based on this semiconcavity property. %
	Then, we use a discrete time stochastic optimal control interpretation of the full discretization scheme, and we establish a convergence rate of order $1$ in terms of both time and spatial steps %
	using certain interpolation operators, under further regularity assumptions. 
	We extend our convergence results to problems with particular state constraints. We apply our results to analyze the convergence rate and computational complexity of the fast-marching method.
	We also consider the multi-level fast-marching method recently introduced by the authors.
\end{abstract}

\maketitle

\section{Introduction}
\subsection{Motivation and Context} 
We consider a %
class of optimal control problems, the minimum time problem, with a given target set. To this problem is associated an eikonal equation, 
which is also a static first order Hamilton-Jacobi(HJ) Partial Differential Equation(PDE).  %
The value function %
is then characterized as the solution of the associated HJ equation in the viscosity sense (see for instance~\cite{flemingsoner}). Problems with state constraints can be addressed with the notion of constrained viscosity solution~\cite{soner1986optimal1}. 

One classical numerical method for approximating HJ equations is the \textit{semi-Lagrangian scheme}, as in \cite{falcone1987numerical,falcone2013semi}, which arises by applying the Bellman dynamic programming principle to the discrete time optimal control problem obtained after a Runge-Kutta time-discretization of the dynamics. 
For infinite horizon discounted problems, assuming strong comparison principle for the associated HJ equation and considering only time discretization, a convergence rate of order $1/2$ in terms of time step is established under mild conditions on the problems, and a convergence rate of order $1$ is established typically under a semiconcavity condition~\cite{dolcetta1984approximate}, or a bounded variational condition~\cite{falcone1987numerical}. 
For practical computation, a further discretization in the state space is needed, which leads to the full discretization semi-Lagrangian scheme. 
After such a space discretization (using a grid), the resulting system of equations can generally (when the discretization satisfies a monotonicity condition) be interpreted as the dynamic programming equation of a stochastic optimal control problem \cite{kushner2001numerical} with discrete time and state space. 
One can then solve the discretized equation by applying value iteration
until convergence. 
Each iteration consists in updating the value function at nodes of a given grid by solving the corresponding discrete HJ equation in these grid nodes. 
The convergence of the full discretization scheme is often obtained as both the time step and the ratio of space step over time step tend to $0$. 
Several further works are proposed intended to show the convergence and compute the convergence rate of this class of schemes, for instance the works of Bardi and Falcone~\cite{bardi1990approximation}, of Falcone and Ferretti~\cite{falcone1994discrete}, of Gr\"une~\cite{grune1997adaptive}, of Bokanowski, Megdich and Zidani~\cite{bokanowski2010convergence}, and the more recent works of ~Bokanowski, Gammoudi and Zidani\cite{bokanowski2022optimistic}, of Calzola, Carlini, Dupuis and Silva~\cite{calzola2023semi} and of de Frutos and Novo~\cite{de2023optimal}.

The value iteration algorithm solving the discretized HJ equation requires generally a number of iterations which depends on the mesh step. 
An interesting acceleration method is the fast-marching method, which was originally introduced in~\cite{tsitsiklis1995efficient} and~\cite{sethian1996fast} in the case of monotone and causal discretizations of the eikonal equation. 
The method takes advantage of the property that the evolution of the region behinds a ``propagation front" is monotonically non-decreasing, allowing one to focus only on the computation around the front at each iteration. 
Specifically, the value function is computed by visiting the grid nodes in a special order, which is chosen so that the value function is monotone non-decreasing in the direction of propagation. 
This is the so-called ``causality''. 
Owing to these properties, the fast-marching method is a ``single pass" method, and it solves the discretized equation exactly. 
Therefore, the convergence of fast marching method is equivalent to the convergence of the discretization.
The computational complexity of the fast-marching method is shown to be $O(K M \log(M))$ in terms of arithmetic operations, where $M$ is the number of grid nodes and $K$ is a constant that depends on the size of discretization neighborhoods. In particular, considering a $d$-dimensional grid with mesh step $h$, the computational complexity is $\widetilde{O}((\frac{1}{h})^d)$, where $\widetilde{O}$ ignores the logarithm factors. 
In first works 
on the fast marching method, by Sethian and Vladimirsky~\cite{sethian2003ordered}, Cristiani and Falcone ~\cite{cristiani2007fast}, Carlini, Falcone, Forcadel and Monneau~\cite{carlini2008convergence}, Carlini, Falcone and Hoch~\cite{carlini2011generalized} and Mirebeau~\cite{Mir14a}, the authors proved the convergence of their methods when the mesh step $h$ goes to $0$ without an explicit convergence rate. 
More recently, Shum, Morris, and Khajepour~\cite{shum2016},
and Mirebeau~\cite{mirebeau2019riemannian}  established a  convergence rate of order $1/2$ in $h$, meaning that the error is  $O(h^{\frac{1}{2}})$. Though, most of numerical experiments in above works reveal an actual convergence rate 
of order $1$. 
One of the objectives of the present paper is to establish sufficient conditions for achieving a convergence rate of order $1$.

The fast-marching method still suffers from the ``curse of dimensionality". 
One possible way to overcome the curse of dimensionality is to focus on finding one (or several) particular optimal trajectories. 
In 
\cite{akian2023multi}, we introduced a multi-level fast-marching method, which 
focuses on the neighborhood of optimal trajectories. 
We also obtained a theoretical computational complexity bound for this method,
under assumptions on the convergence rate of the discretization (without or with a particular state constraint) and on the stiffness of the value function.
In particular, the best complexity bound is of the same order as for one-dimensional problems and is obtained when the convergence rate is $1$ and the stiffness is high. One of the aim of the present paper is to give sufficient conditions on the problem to achieve such a 
convergence rate. 

\subsection{Contribution}

We consider an eikonal equation arising from the minimum time problem of reaching a target set $\K$, and 
a particular semi-Lagrangian scheme 
in which the time step depends on the state so that the discrete time problem becomes a minimal cost problem with constant speed.
Considering first the time discretization (only), we show in~\Cref{semiconcave-value} that, under certain regularity assumptions on the dynamics and on the boundary of target set $\K$, the discrete time value function is semiconcave. 
This regularity of the discrete time value function leads to a convergence rate of order $1$ in the time step of the semi-Lagrangian scheme, which is stated in~\Cref{rate_main}. 

Then, we consider a full discretization scheme that involves discretizing the state space using a mesh grid, and is derived by further applying certain interpolation operators. In this scheme, we take comparable values of space and time steps, which implies that it does not satisfy the typical assumptions for convergence in the existing literature.
Using the regularity obtained for the discrete time value function and assuming additional regularity conditions on the dynamics and on the boundary of the target set $\K$, we prove that the error between the solution of the fully discretized equation and the solution of the semi-discretized equation is of order $1$ in terms of the mesh step. 
In particular, in~\Cref{rate_wh-v} we show that the error bound holds in one direction, by the contraction property of the update operator. In~\Cref{rate_v-wh}, we show that the error bound holds in the other dirertion, using a controlled Markov problem interpretation of the full discretization scheme.  
This result yields a convergence rate of order $1$ for the full discretization scheme, in terms of both time step and mesh step. 

As an application of the above results, we 
obtain a convergence rate of order 1 for the fast-marching methods
used in the works of~\cite{sethian2003ordered,cristiani2007fast,shum2016, mirebeau2019riemannian}. Therefore they match the numerical experiments of the aforementioned works.

We also extend our convergence results to problems involving certain types of state constraints and show in that case that the constants of the errors are independent of the state constraint, for points which are far from the boundary.

This paper is organized as follows: In~\Cref{sec-pre}, we provide preliminary results on the HJ equation and the minimum time optimal control problem. In~\Cref{sec-semilagrangian}, we consider a discrete time optimal control problem associated to the semi-Lagrangian time discretization scheme. A convergence rate of order $1$ is obtained using the semiconcavity property of the value function of the discrete time problem, which itself is obtained under semiconcave assumptions on the dynamics and target set. In~\Cref{sec-fulldisfm}, we consider a full discretization scheme. We represent the solution of the fully discretized equation as the value function of a stochastic control problem (or Markov Decision problem). We then show the convergence rate for particular interpolation operators. As an application we analyze the computational complexity of a fast-marching method with update operator derived from the full discretization scheme. In~\Cref{sec-mlfm}, we extend convergence results to a particular type of state constraints, and then apply the results to analyze the computational complexity of the multi-level fast-marching method.

\paragraph{\bf Acknowledgement:} The authors thank St\'ephane Gaubert for his useful and detailed  suggestions.

\section{Preliminaries}\label{sec-pre}
\subsection{The Eikonal Equation}
Let $\K$ be a compact subset of $\R^{d}$. Let $S_{1}$ be the unit Euclidean sphere in $\R^{d}$, i.e., $S_{1} = \{ x \in \R^{d}, \| x\| =1  \} $ where $\|\cdot\|$ denotes the Euclidean norm.
We consider an eikonal equation of the form:
\begin{equation}\label{eikonal_eq}
	\left\{
	\begin{aligned}
		&-( \min_{\alpha \in S_1} \{ (\nabla T(x) \cdot \alpha) f(x,\alpha) \} + 1 ) = 0, \quad &x \in \R^d \setminus \K \ , \\
		&T(x) = 0, \quad & x \in \partial \K \ ,
	\end{aligned}
	\right.
\end{equation} 
where $f$ is the speed function, and we assume the following basic regularity properties:
\begin{assumption} \label{assp1}$ $
	\begin{enumerate}
		\item $f : \R^d \times S_1 \mapsto \R_{>0}$  is continuous.	
		\item $f$ is bounded, i.e., $\exists M_f >0$ s.t. $|f(x,\alpha) | \leq M_f$, $\forall x \in \R^d, \forall \alpha \in S_1$.
		\item There exists constants $L_{f}, L_{f,\alpha}>0$ such that $ |f(x,\alpha) - f(x^{'},\alpha)| \leq L_{f}| x - x^{'} |, \forall \alpha \in S_1, \forall x, x^{'}\in \R^d$ and  $| f(x,\alpha) - f(x,\alpha^{'})| \leq L_{f,\alpha}|\alpha - \alpha^{'}|, \forall x \in \R^d, \forall \alpha, \alpha^{'} \in S_1$.
	\end{enumerate}
\end{assumption}
The function $T: \R^d \to \R$ represents the minimum time required for a point $x \in \R^d \setminus \K$ to reach $\K$, while traveling with a state-dependent speed given by the function $f$. Such an eikonal equation is typically associated with the front propagation problem,  which involves the evolution of the boundary of a domain, denoted by $\Gamma_t$, as described by $T$.
In particular, the boundary of the domain $\Omega_t$ can be defined as $\Gamma_t = \partial \Omega_t = \{x\in\mathbb{R}^d \mid T(x) = t\}$, where the initial condition is $\Omega_0 = \mathcal{K}$. Notice that, given~\Cref{assp1}, we have $\Omega_t \subset \Omega_{t+s}$ for all $t, s > 0$.

\subsection{Minimum Time Optimal Control Problem}\label{subsec-minimumtime} The above equation~\eqref{eikonal_eq} also arises from the minimum time problem. A basic technique in the study of this problem (see for instance \cite{vladimirsky2006static}, \cite{bardi1989boundary}, \cite[Chapter-IV]{bardi2008optimal}) is the change of variable:
\begin{equation} \label{changevar}
	v(x) = 1 - e^{-T(x)} \enspace,
\end{equation}
which was first used by Kruzkov~\cite{kruvzkov1975generalized}. By doing so, $v(x)$ is automatically bounded and Lipschitz continuous. Once $v$ is computed, one can directly compute the value of $T(x)$ by $T(x) = - \log(1 - v(x))$.  

Let us consider a control problem associated to the dynamical system:
\begin{equation}
	\label{dynmsys}
	\left\{
	\begin{aligned}
		&\dot{y} (t) = f(y(t), \alpha(t)) \alpha(t), \ \forall t \geq 0 \ ,\\
		&y(0) = x \ ,
	\end{aligned}
	\right.
\end{equation}
where $\alpha \in \mathcal{A} := \{ \alpha : \R_{\geq0} \to S_{1}, \ \alpha(\cdot) \text{ is measurable} \}  $. Every $\alpha \in \A$ is then the unit vector determining the direction of motion.
We denote by $y_\alpha(x;t)$ the solution of the dynamical system~\eqref{dynmsys} for a given $\alpha ( \cdot)$, and define a discounted cost functional by:
\begin{equation}\label{cost}
	J(\alpha (\cdot),x) = \inf \left\{ \int_{0}^{\tau} e^{-t} dt \mid \tau\geq 0 , \ y_{\alpha}(x;\tau) \in \K \right\} \enspace ,
\end{equation}
for $\alpha \in \A$. Then, the value function $v$ of the control problem given by
\begin{equation}\label{value}
	v(x) = \inf_{\alpha \in \A} J(\alpha(\cdot),x)
\end{equation}
coincides with $v$ in~\eqref{changevar}. Let now 
\begin{equation}
	\label{defF}
	F(x,r,p) = -\min_{\alpha \in \A} \{ p \cdot f(x,\alpha)\alpha + 1 - r \}
	\enspace .\end{equation}
This Hamiltonian corresponds to the  control problem {\rm (\ref{dynmsys},\ref{cost},\ref{value})}. Then, under~\Cref{assp1}, restricted to $\overline{\R^d \setminus \K}$, $v$ is the unique viscosity solution of the following Hamilton-Jacobi-Bellman equation (see for instance \cite{flemingsoner}):
\begin{equation}\label{HJ}
	\left\{
	\begin{aligned}
		& F(x, v(x), Dv(x)) = 0, \enspace & x \in \R^d \setminus \K  \ , \\
		&v(x) = 0, \enspace & x \in \partial \K \ .
	\end{aligned}
	\right.
\end{equation}
Let $\Oc$ be an open subset of $\R^d$, we also briefly recall the definition of viscosity solution of 
\begin{equation}\label{eq_def}
	F(x, u, D u) =0, \ \text{ in } \Oc \ , 
\end{equation}
\begin{definition}\label{def_viscosity}
	Let $u: \Ocb \to \R$ be continuous.
	\begin{enumerate}
		\item\label{vis_sub} The function $u$ is a viscosity subsolution of~\eqref{eq_def} if for every test function $\psi \in \Cc^1(\Ocb)$ and all local maximum points $x_0 \in \Ocb$ of the function $u - \psi$, we have $F(x_0, u(x_0), D\psi(x_0))\leq 0$. 
		\item\label{vis_super}  The function $u$ is a viscosity supersolution of~\eqref{eq_def} if for every test function $\psi \in \Cc^1(\Ocb)$ and all local minimum points $x_0 \in \Ocb$ of the function $u - \psi$, we have $F(x_0, u(x_0), D\psi(x_0))\geq 0$. 
		\item The function $u$ is a viscosity solution of~\eqref{eq_def} if and only if it is a viscosity subsolution and supersolution of~\eqref{eq_def}. 
	\end{enumerate}
	
\end{definition}

In the following, we will focus on the numerical approximation of system~\eqref{HJ}.

\section{The Semi-Lagrangian Scheme: Convexity Properties And Convergence Analysis.}\label{sec-semilagrangian}

In this section, we propose a semi-Lagrangian type discretization (in time) for the system~\eqref{HJ}. 
We analyze the convergence of the solution of discretized equation to the viscosity solution of~\eqref{HJ}, and we give the convergence rate. 
A convergence rate of oder $1$ in terms of time step is also established by exploiting the semiconcavity property of the solution of discretized equation, which is associated with the value function of a discrete time optimal control problem.

\subsection{The Semi-lagrangian Scheme for the Minimum Time Problem}
Consider the following semi-Lagrangian type discretization of the system~\eqref{HJ}:
\begin{equation}\label{discre-hjb}
	\left\{
	\begin{aligned}
		&v^h(x) = \min_{\alpha \in S_{1}} \left\{ \big(1-\frac{h}{f(x,\alpha)}\big) v^{h}(x + h \alpha)  + \frac{h}{f(x,\alpha)}   \right\} , \ & x \in \R^d \setminus \K \  ,\\
		&v^h(x) = 0, \ &x \in  \K \ ,
	\end{aligned}
	\right.
\end{equation}
where $h>0$ is a fixed parameter. 
This is a direct discretization in time of system~\eqref{HJ}, in which the time step is $h/f(x,\alpha)$, depending on state and control. 
The convergence of similar discretization systems, for which the time step is constant, has been studied for instance in~\cite{bardi1990approximation,falcone1994level}, and the method of proof can be straightforwardly
adapted to our system~\eqref{discre-hjb}, keeping
in mind that~\eqref{HJ} has a unique viscosity solution $v$.
\begin{proposition}\label{convergence}
	Let us denote 
	$ \supv(x) = \liminf \limits_{h \to 0, \  y \to x } v^h (y)$, and  $\subv(x) =  \limsup \limits_{h \to 0, \ y \to x} v^h(y)$. Make~\Cref{assp1}, then $\subv$ ($\supv$ resp.) is a viscosity subsolution (supersolution resp.) of \eqref{HJ}.
	Thus, $\{v^h\}$ converge uniformly to $v$ on any compact subset of $\R^d$ as $h \to 0$. \qed
\end{proposition}
In the following, we denote $\upperboundf$ and $\lowerboundf$ the upper and lower bounds for $f$, respectively, i.e., 
\[0 < \lowerboundf \leq f(x,\alpha) \leq \upperboundf < \infty, \text{ for all } x \in \R^d \text{ and } \alpha \in \A \ . \]
Then, we have the following result for the convergence rate.
\begin{proposition}\label{converg}
	Suppose that~\Cref{assp1} holds. 
	There exists a constant $C_{1/2}>0$ depending on $L_f, L_v, \upperboundf$  such that, for every $0<h<\frac{1}{\upperboundf} $:
	\begin{equation}\label{rate_1}
		\|v^h - v \|_\infty \leq C_{1/2} h^{\frac{1}{2}} \ .
	\end{equation}
\end{proposition}
\begin{proof}
	The proof for \Cref{converg} is a slight modification of the original method in~\cite{Crandall1984TwoAO}. Therefore, we will only provide a brief sketch of the proof here, with the purpose of facilitating further analysis. 
	
	Let us denote by $\Omega = \R^d \setminus \K$, 
	and define the following series of auxiliary functions.
	For every $1>\varepsilon>0$, for every $x \in \Omega ,$ we set $ \theta_{\varepsilon}(x) = - |\frac{x}{\varepsilon}|^2.$
	For every $0<h<\frac{1}{\upperboundf}$, for every $(x,y) \in \Omega \times \Omega,$ we set
	$\varphi(x,y) = v^{h}(x) - v(y) + \theta_{\varepsilon}(x-y) .$ 
	As both $v^h$ and $v$ are bounded, for every $\zeta > 0$, there exists a point $(x_1,y_1) \in \Omega \times \Omega$ which is an approximate maximizer
	of $\varphi$ up to a margin $\zeta$, i.e.,
	\[\varphi(x_1, y_1) > \sup_{(x,y) \in \Omega \times \Omega} (\varphi (x,y) - \zeta) \ .\]
	Let us choose a function $\xi \in C_0^{\infty}(\Omega \times \Omega)$, such that
	$\xi(x_1,y_1) = 1$, and $\xi \in [0,1], \ |D\xi| \leq 1 .$
	For every $1 > \zeta >0$, for every $(x,y) \in \Omega \times \Omega$, let 
	$\psi(x,y) = \varphi(x,y) + \zeta \xi(x,y).$ 
	Let $(x_0,y_0)$ be the point where $\psi$ reaches its maximum, i.e, 
	\begin{equation}\label{max_psi}
		\psi(x_0,y_0) \geq \psi(x,y),\ \text{for all } (x,y) \in \Omega \times \Omega \ . 
	\end{equation}
	Then, automatically $y \to -\psi(x_0,y) = v(y) - ( v^h(x_0)+ \theta_{\varepsilon}(x_0-y)+\zeta \xi (x_0,y)) $ reaches its minimum at $y_0$. By definition of viscosity solution, letting $y \to (v\fromsource^h(x_0) + \theta_{\varepsilon}(x_0-y) + \zeta \xi(x_0,y))$ be a test function, we have:
	\begin{equation}\label{2vhy}
		v(y_0) -  ((D\theta_{\varepsilon}(x_0-y_0) \cdot \alpha^{*}- \zeta D_y\xi(x_0,y_0))\cdot \alpha^{*}) f(y_0,\alpha^{*}) - 1 \geq 0 \enspace,
	\end{equation}
	for some $\alpha^{*} \in S_1$. Since $v^h$ is the solution of system~\eqref{discre-hjb}, we have
	\begin{equation}\label{upbondvh}
		v^h(x_0) \leq \left\{ (1-\frac{1}{f(x_0,\alpha^*)}) v^h(x_0+h \alpha^*) + \frac{h}{f(x_0,\alpha^*)} \right\} \ .
	\end{equation}
	Take $x = x_0 + h \alpha^*, y= y_0$ in~\eqref{max_psi}, we get
	\begin{equation}\label{upboundvha} 
		v^h(x_0 + h \alpha^*) \leq v^h(x_0) + (D \theta_\varepsilon (x_0 - y_0) \cdot \alpha^*)h + \frac{2}{\varepsilon^2} \alpha^* h^2 + \zeta \alpha^* h \ .
	\end{equation}
	Combining~\eqref{upbondvh} and~\eqref{upboundvha}, we get
	\begin{equation}\label{first_approx}
		v^h(x_0) \leq (1 - \frac{h}{f(x_0,\alpha^*)}) ( (D \theta_\varepsilon(x_0-y_0)  \cdot \alpha^* ) + \frac{2}{\varepsilon^2} \alpha^* h + \zeta \alpha^*)f(x_0,\alpha^*) + 1 \ .  
	\end{equation}
	Combining~\eqref{2vhy} and~\eqref{first_approx}, we have
	\begin{equation}\label{second_approx}
		v^h(x_0) - v(y_0) \leq \frac{2 L_f |x_0 -y_0 |^2 }{\varepsilon^2} + \frac{2 |x_0 - y_0 |h}{\varepsilon^2} + \frac{2 \upperboundf h}{\varepsilon^2} + 2 \zeta \upperboundf \ .
	\end{equation}
	Let us choose $x = y = x_0$ in~\eqref{max_psi}, then we obtain 
	\begin{equation}\label{bound_x0y0}
		| x_0 -y_0 | \leq (L_v + \zeta) \varepsilon^2 \ ,
	\end{equation}
	where $L_v$ is the Lipschitz constant for $v$. Substituting~\eqref{bound_x0y0} into~\eqref{second_approx}, we have
	\begin{equation}\label{third_approx}
		v^h(x_0) - v(y_0) \leq 2L_f(L_v+\zeta)^2\varepsilon^2+2(L_v+\zeta)h+\frac{2 \upperboundf h }{\varepsilon^2} + 2 \zeta \upperboundf \ .
	\end{equation}
	Take $\varepsilon = h^{\frac{1}{4}}$, we get
	\begin{equation}\label{4_approx}
		v^h(x_0) - v(y_0) \leq (2L_f(L_v+\zeta)^2 +2 \upperboundf) h^{\frac{1}{2}} + 2(L_v+\zeta)h + 2 \zeta \upperboundf \ .
	\end{equation}
	Let us now choose $x = y$ in~\eqref{max_psi}, we obtain that 
	\begin{equation}\label{5_approx}
		v^h(x) - v(x) \leq v^h(x_0) - v(y_0) + \zeta ( \xi(x_0,y_0) - \xi (x,x) ) - \frac{|x_0 - y_0 |^2}{ \varepsilon^2} \ .
	\end{equation}
	Thus, combining~\eqref{4_approx} and~\eqref{5_approx}, and take $\zeta \to 0$, we obtain that
	\begin{equation}\label{final_approx_one}
		v^h(x) - v(x) \leq ( 2 L_f L_v^2 + 2 \upperboundf  ) h^{\frac{1}{2}} \ .
	\end{equation}
	To show $v(x) - v^h(x) \leq ( 2 L_f L_v^2 + 2 \upperboundf  ) h^{\frac{1}{2}} $, it is enough to take $\varphi(x,y) = v(x) - v^h(y) + \theta_\varepsilon(x-y)$. We conclude the estimate in~\Cref{converg} with $C_{1/2} = 2 L_f L_v^2 + 2 \upperboundf$.
\end{proof}

\subsection{Discrete Time Control Problem and Its Value Function}\label{subsec_discre_control}

Let us represent the solution of the discretized equation~\eqref{discre-hjb} as the value function of a discretized version of the control problem~{\rm (\ref{dynmsys},\ref{cost},\ref{value})}. 
Consider the following discrete dynamical system,
\begin{equation}\label{discre_dyn}
	\left\{
	\begin{aligned}
		&y^h(k+1) = y^h(k) + h \alpha_k, \  \forall k = 0,1,2,\dots \ ,\\
		&y^h(0) = x \ ,
	\end{aligned}
	\right.
\end{equation}
where $\alpha_k \in S_1 $, for every $k = 0,1,2,\dots$ . Let us simply denote $\alpha^h$ the sequence of controls $\{ \alpha_k \}_{k=0,1,2,\dots}$, and denote $y^h_{\alpha^h}(x;k)$, $k = 0,1,2,\dots$, the solution of the above system \eqref{discre_dyn} with a given sequence of controls $\alpha^h$. Moreover, let
\begin{equation}\label{dis_step}
	N (x,\alpha^h) = \inf \{ N \in \mathbb{N}_+ \mid y^h_{\alpha^h}(x;N) \in \K \}\ .
\end{equation}
Consider the following discrete cost functional:
\begin{equation}\label{dis_cost}
	J^h(\alpha^h,x) =  \sum_{k=0}^{N(x,\alpha^h)} \Big( \frac{h}{f(y^h_{\alpha^h}(x;k),\alpha_k)} \big( \prod_{l=0}^{k-1}(1-\frac{h}{f(y^h_{\alpha^h}(x;l),\alpha_l)})  \big) \Big) \ .
\end{equation}
The associated value function is given by
\begin{equation}\label{value_discret}
	v^h(x) = \inf_{\alpha^h \in \A^h} J^h(\alpha^h,x) \ ,
\end{equation}
where $\A^h$ is a subset of $\A$ containing the controls which take constant values in the interval $[k, k+1[$, for every $k = 0,1,2,\dots ,$ i.e., 
\begin{equation}	\A^h = \{ \{ \alpha_k \}_{k\geq 0} \mid \alpha_k \in S_1 , \ \forall k=0,1,2,\dots \ \} \ .
\end{equation}
Then, the value function of this discrete optimal control problem is the solution of equation~\eqref{discre-hjb} (See for instance~\cite{flemingsoner,bardi2008optimal}). 

Note that an equivalent formulation of the discrete cost functional in~\eqref{dis_cost} is given by
\begin{equation}\label{discrete_cost_eq}
	J^h(\alpha^h,x) = 
	1- \prod_{k=0}^{N(x,\alpha^h)}\left(1-\frac{h}{f(y^h_{\alpha^h}(x;k),\alpha_k)}
	\right)
	\ .
\end{equation}
The equality follows from an elementary computation. 
We will use this formulation of $v^h$ in the following.

\subsection{Order $1$ Convergence Rate Under A Semiconcavity Assumption}

Let us denote 
\begin{equation}
	d_{\mathcal{K}}(x):= \inf_{y \in \mathcal{K}} \|y-x \| \ , \ \text{ for every } x \in \R^d \ , 
\end{equation} the distance function from $x$ to the target set $\K$. We consider the following assumptions on the target set $\K$ and speed function $f$.
\begin{assumption}\label{assump_2} $ $
	\begin{enumerate}
		\item\label{assump_2_1} There exists a constant $M_f>0$ such that
		\begin{equation}\label{semiconcav_f}
			\frac{1}{f(x+z,\alpha)} - 2\frac{1}{f(x,\alpha)} + \frac{1}{f(x-z,\alpha)}  \leq M_f |z|^2, \ \forall x,z \in \R^d, \  \forall \alpha \in S_1 . \end{equation}
		\item\label{assump_2_2} There exists a constant $M_t>0$ such that
		\begin{equation}\label{semiconcave_dst}
			d_{\K} (x+z) + d_{\K}(x-z) -2 d_{\K}(x)  \leq M_t |z|^2, \ \forall x,z \in \R^d  \setminus \mathring{\K} \ .\end{equation}
	\end{enumerate}
\end{assumption}
The assumption presented in~\eqref{assump_2_1} of~\Cref{assump_2} addresses the semiconcavity property of the inverse of the speed function. 
In the following, we give specific criteria for verifying~\eqref{assump_2_2} in~\Cref{assump_2}. 
This condition appeared in~\cite{cannarsa1995convexity} and~\cite{cannarsa2004semiconcave}, and is regarded as a semiconcavity property for the distance function $d_{\K}$ in $\R^d \setminus \mathring{\K}$. The authors of~\cite{cannarsa1995convexity}
have provided sufficient conditions for checking \eqref{semiconcave_dst}, which we present as a lemma: 
\begin{lemma}[Corollary of \protect{\cite[Prop. 3.2]{cannarsa1995convexity}}]\label{ness_semiconcave}
	If there exists $r_t>0$ such that
	\begin{equation} \label{con_semi_dst}
		\forall x \in \K, \ \exists x_0\in \K : x \in \overline{B^d(x_0,r_t)} \subset \K \ , 
	\end{equation}
	then \eqref{semiconcave_dst} holds. In particular, if $\partial \K$ is of class $\Cc^{1,1}$, then  \eqref{semiconcave_dst} holds.\hfill\qed
\end{lemma}

\begin{theorem}[Semiconcavity of discrete time value function]\label{semiconcave-value}
	Suppose that~\Cref{assp1} and~\Cref{assump_2} hold. Then, we have 
	\begin{equation}\label{uhorder1}
		v^h(x+z) - 2v^{h}(x) + v^h(x-z) \leq C_v |z|^2 , \text{ for every } x, z \in \R^d \setminus \K \ ,
	\end{equation}
	where $C_v$ is a constant depends on $M_f, M_t, \upperboundf, \lowerboundf$.
	
\end{theorem}
\begin{proof}Let us denote $\alpha^*_x = \{\alpha^*_{x,0}, \alpha^*_{x,1}, \alpha^*_{x,2},\dots,\alpha^*_{x,N_x} \}$ the discrete optimal control for which the infimum in \eqref{value_discret} is obtained, and let us simply denote $N_x = N(x,\alpha^*_x)$. For the problem starting from $(x+z)$ ($x-z$ resp. ), let us consider a control $\alpha'_{x+z}$ ($\alpha'_{x-z}$ resp. ) defined as following: $\alpha'_{x+z}$ ($\alpha'_{x-z}$ resp. ) takes the same control as $\alpha^*_x$ until one of the three trajectories $y^h_{\alpha^*}(x;\cdot)$, $y^h_{\alpha'_{x+z}}(x+z;\cdot)$ and $y^h_{\alpha'_{x-z}}(x-z;\cdot)$ reaches $\K$. Then, we consider two cases.
	
	\textbf{Case 1.} $\forall N\leq N_x, \ y^h_{\alpha'_{x+z}}(x+z;N) \notin \K $ and $ y^h_{\alpha'_{x-z}}(x-z;N) \notin \K$. In this case, the optimal trajectory for the problem starting from $x$ will first reach $\K$. Then, for the problem starting from $x+z$ ($x-z$ resp. ), we take the control following the shortest distance path, in euclidean sense, from $y^h_{\alpha'_{x+z}}(x+z; N_x)$ ($y^h_{\alpha'_{x-z}}(x-z; N_x)$ resp. ) to $\K$. Let $N_+$, $N_-$ denote the steps for which $y^h_{\alpha'_{x+z}}(x+z; N_+), y^h_{\alpha'_{x-z}}(x-z; N_-) \in \K$. For easy expression, we simply denote $y^h_{x,m} := y^h_{\alpha^*_x}(x;m)$, $y^h_{+,m} := y^h_{\alpha'_{x+z}}(x+z; m)$, $y^h_{-,m} := y^h_{\alpha'_{x-z}}(x-z; m)$ for $m = 0,1,2,\dots$ and $\alpha_+ = \alpha'_{x+z}$, $\alpha_- = \alpha'_{x-z}$.
	
	Following \eqref{discrete_cost_eq}, we have:
	\begin{equation}\label{case_1}
		\begin{aligned}
			&v^h(x+z) - 2 v^h(x) + v^h(x-z) \\
			& \leq \{ J^h(\alpha_+,x+z) + J^h(\alpha_-,x-z) - 2 J^h(\alpha^*_x,x) \} \\
			& \leq \left\{ \Big(1 - \prod_{k=0}^{N_+} (1-\frac{h}{f(y^h_{+,k},\alpha_+)})\Big) + \Big(1 - \prod_{k=0}^{N_-} (1-\frac{h}{f(y^h_{-,k},\alpha_-)})\Big) \right. \\ 
			&\qquad \left.  -2\Big(1 - \prod_{k=0}^{N_x} (1-\frac{h}{f(y^h_{x,k},\alpha_x^*)})\Big) \right\} \\
			&\leq\prod_{k=0}^{N_x} (1-\frac{h}{f(y^h_{x,k},\alpha_x^*)}) \left\{ \Big( 1- \frac{\prod_{k=0}^{N_+} (1-\frac{h}{f(y^h_{+,k},\alpha_+)})}{\prod_{k=0}^{N_x} (1-\frac{h}{f(y^h_{x,k},\alpha_x^*)})} \Big) + \Big( 1 - \frac{\prod_{k=0}^{N_-} (1-\frac{h}{f(y^h_{-,k},\alpha_-)})}{\prod_{k=0}^{N_x} (1-\frac{h}{f(y^h_{x,k},\alpha_x^*)})} \Big) \right\}  \\
			&\leq \left[  (1-\frac{h}{\upperboundf})^{N_x} \left\{ \sum_{k=0}^{N_x} \left( \frac{h}{f(y^h_{+,k},\alpha^*_x)} + \frac{h}{f(y^h_{-,k},\alpha^*_x)} - 2\frac{h}{f(y^h_{x,k},\alpha_x^{*})}\right) \right\} \right]  \\ 
			& \quad  +\left[  (1-\frac{h}{\upperboundf})^{N_x}  \left\{ \sum_{k=N_x}^{N_{x_+}}\frac{h}{f(y^h_{+,k},\alpha_+)}  +   \sum_{k=N_x}^{N_{x_-}}  \frac{h}{f(y^h_{-,k},\alpha_-)} \right\} \right] \ .
		\end{aligned}
	\end{equation} 
	Notice that in first $N_x$ steps, three dynamics take the same control $\alpha^*_x$. 
	Let us first focus on the first part inside of $[ \ ]$ in~\eqref{case_1}, for which we denote by $\Delta_1$. Denote for simplification $A_k: =  \frac{1}{f(y^h_{+,k},\alpha_x^{*})} + \frac{1}{f(y^h_{-,k},\alpha_x^{*})} - 2\frac{1}{f(y^h_{x,k},\alpha_x^{*})}$, then
	\begin{equation}
		\begin{aligned}
			A_k &= \frac{1}{ f( y^h_{x,k} + ( y^h_{+,k} -y^h_{x,k}) ,\alpha_x^{*})  } - 2\frac{1}{f(y^h_{x,k},\alpha_x^{*})} + \frac{1}{f( y^h_{x,k} - ( y^h_{+,k} -y^h_{x,k})  ,\alpha_x^{*} )} \\
			&\quad + \frac{1}{f(y^h_{-,k},\alpha_x^{*})} -  \frac{1}{f( y^h_{x,k} - ( y^h_{+,k} -y^h_{x,k})  ,\alpha_x^{*} )} \\ 
			&\leq M_f | y^h_{+,k} -y^h_{x,k} |^2 + \frac{L_f}{\lowerboundf^2} | y^h_{+,k} - 2y^h_{x,k} + y^h_{-,k} | \ .
		\end{aligned}
	\end{equation}
	By the above construction of $\alpha_+$ and $\alpha_-$, we notice that $|y^h_{+,k} - y^h_{x,k}| = z$ and $| y^h_{+,k} - 2y^h_{x,k} + y^h_{-,k} | = 0$, for all $k \in \{0,1,2,\dots, N_x \}$. Thus $A_k \leq M_f z^2$. Then we have:
	\begin{equation}\label{firstpart}
		\Delta_1 \leq M_f |z|^2 \sum_{k=0}^{N_x} (1-\frac{h}{\upperboundf})^{N_x} h \leq M_f \upperboundf |z|^2 \ .
	\end{equation}
	For the second part inside of $[ \ ]$ in \eqref{case_1}, for which we denote by $\Delta_2$, we notice that at the end of $N_x$ step, $y^h_{x,N_x}\in \K$, $y^h_{+,N_x} = y^h_{x,N_x} + z$, $y^h_{-,N_x} = y^h_{x,N_x} - z$. Thus, by \eqref{semiconcave_dst}, we have:
	\begin{equation}
		d_{\K}( y^h_{+,N_x} ) + d_{\K}(y^h_{-,N_x}) \leq M_t |z|^2 \ .
	\end{equation}
	Hence we have
	\begin{equation}\label{secondpart}
		\Delta_2 \leq \frac{M_t}{\lowerboundf} |z|^2 \ .
	\end{equation}
	Combine \eqref{firstpart} and \eqref{secondpart}, we deduce \eqref{uhorder1} with $C_1 = M_f \upperboundf +  \frac{M_t}{\lowerboundf} $. 
	
	\textbf{Case 2.} $\exists N_- \leq N_x$ such that $y^h_{\alpha'_{x-z}}(x-z;N_-) \in \K$. In this case, the optimal trajectory for the discrete time problem starting from $x-z$ first reaches $\K$. Then, for the problem starting from $x+z$, let us consider a control $\alpha'_{x+z}$ defined as follows:
	\begin{equation}\label{construct_control}
		\alpha'_{x+z,k} = \left\{
		\begin{aligned}
			& \alpha^*_{x,k} \ , \quad & k \in \{0,1,2,\dots, N_-\} \ ; \\
			& \alpha^*_{x, (N_- + \lfloor \frac{k - N_-}{2} \rfloor) } \ ,  \quad & k \in \{ N_-,(N_- +1),\dots, (2N_x - N_- )\} \ ; \\
			&\text{following Euclidean shortest path, } &  (2N_x - N_-)< k \leq N_+ \ ,
		\end{aligned}
		\right.
	\end{equation}
	with $y^h_{\alpha'_{x+z}}(x+z;N_+) \in \K$. We argue as $N_+ \geq (2N_x - N_- )$, then the result will automatically holds as it is in a weaker situation when $N_+ < (2N_x - N_- )$. We also take the same simplified notations as in Case 1., and we omit the same computations. Then we have
	\begin{equation}\label{case_2}
		\begin{aligned}
			&v^h(x+z) - 2 v^h(x) + v^h(x-z) \\
			& \leq \{ J^h(\alpha_+,x+z) + J^h(\alpha_-,x-z) - 2 J^h(\alpha^*_x,x) \} \\
			& \leq \left[ (1-\frac{h}{\upperboundf})^{N_x} \left\{ \sum_{k=0}^{N_-} \left( \frac{h}{f(y^h_{+,k},\alpha^*_x)} + \frac{h}{f(y^h_{-,k},\alpha^*_x)} - 2\frac{h}{f(y^h_{x,k},\alpha_x^{*})}\right) \right\} \right]  \\  
			& \quad + \left[   \sum_{k=(N_- + 1)}^{(2N_x - N_-)}\Big( \frac{h}{f(y^h_{+,k},\alpha_+)} \Big) - 2 \sum_{k = (N_- + 1)}^{N_x} \Big( \frac{h}{f(y^h_{x,k}, \alpha^*_x)} \Big)  \right] \\
			& \quad+ \left[ \sum_{2N_x-N_- +1}^{N_+}\Big( \frac{h}{f(y^h_{+,k},\alpha_+)} \Big) \right] \ .
		\end{aligned}
	\end{equation}
	The first part inside of $[ \ ]$ in~\eqref{case_2} follows the same computation as in Case 1. Let us now focus on the second part inside of $[ \ ]$ in~\eqref{case_2}, for which we denote by $\Delta'_2$. Based on the above construction of $\alpha_-$, we have  $y^h_{-,N_-} = y^h_{x,N_-} - z$. Since $y^h_{-,N_-} \in \K $, we have $d_{\K} (y^h_{x,N_-} ) \leq z$. Since $\alpha^*_x$ is the optimal control for the discrete time problem starting from $x$, we have
	\begin{equation}
		\sum_{k = (N_- + 1)}^{N_x} \frac{| y^h_{x,k} - y^h_{x,k-1} |}{f(y^h_{x,k-1},\alpha^*_x)} \leq \frac{d_{\K}(y^h_{x,N_-})}{\lowerboundf} \leq \frac{|z|}{\lowerboundf} \ ,
	\end{equation}
	which implies
	\begin{equation}
		\sum_{k = (N_- + 1)}^{N_x} |y^h_{x,k} - y^h_{x,k-1}  | \leq \frac{\upperboundf}{\lowerboundf} |z| \ .
	\end{equation}
	By the above construction of $\alpha_+$ in~\eqref{construct_control}, we have $y^h_{+,N_-} = y^h_{x,N_-} + z $. Moreover, for every $j \in \{ 1,2,\dots, (N_x - N_-) \}$, we have:
	\begin{equation}\label{distance_+x}
		\max \{|y^h_{+, (N_- +2j-1)} - y^h_{x,(N_- +j)} |, \ |y^h_{+,(N_- + 2j)} - y^h_{x,(N_- + j)}| \} \leq (\frac{\upperboundf}{\lowerboundf} + 1) |z| \ . 
	\end{equation}
	Then, by the Lipschitz continuity of $f$, and the fact that $\alpha_{+,(N_- + 2j-1)} = \alpha_{+,(N_- +2j)} = \alpha^*_{x,N_- + j}$, we have:
	\begin{equation}\label{Delta2p}
		\begin{aligned}
			\Delta'_2 &\leq \frac{h}{\lowerboundf^2} \sum_{j=1}^{N_x-N_-} \Big( |f(y^h_{+, (N_- +2j-1)} ,\alpha_+) - f( y^h_{x,(N_- +j)},\alpha^*_x )| \\
			&\qquad+  | f(y^h_{+,(N_- + 2j)} ,\alpha_+  ) - f( y^h_{x,(N_- +j)},\alpha^*_x ) |   \Big) \\
			&\leq \frac{2 L_f \upperboundf}{\lowerboundf^3}(\frac{\upperboundf}{\lowerboundf}+1) |z|^2 \ .
		\end{aligned}
	\end{equation} 
	For the third part inside of $[ \ ]$ in~\eqref{case_2}, for which we denote by $\Delta'_3$, notice that 
	\begin{equation}
		y^h_{+,(2N_x -N_-)} - y^h_{x,N_x} = y^h_{x,N_x} - y^h_{-,N_-} \leq (\frac{\upperboundf}{\lowerboundf} + 1) |z| \ . 
	\end{equation}
	Then, by \eqref{semiconcave_dst}, and the fact that $y^h_{x,N_x} \in \K$, $y^h_{-,N_-} \in \K$, we have
	\begin{equation}
		d_{\K}(y^h_{+,(2N_x-N_-)}) \leq M_t (\frac{\upperboundf}{\lowerboundf} + 1)^2 |z|^2 \ .
	\end{equation}
	Thus
	\begin{equation}\label{Delta3p}
		\Delta'_3 \leq \frac{M_t}{\lowerboundf}(\frac{\upperboundf}{\lowerboundf}+1)^2|z|^2 \ .
	\end{equation}
	Combine \eqref{Delta2p} and \eqref{Delta3p}, we deduce \eqref{uhorder1} with $C_2 = M_f \upperboundf + \frac{2 L_f \upperboundf}{\lowerboundf^3}(\frac{\upperboundf}{\lowerboundf}+1) + \frac{M_t}{\lowerboundf}(\frac{\upperboundf}{\lowerboundf}+1)$. 
	
	Since another possible case, that is $\exists N_+ \leq N_x$ such that $y^h_{\alpha'_{x+z}}(x+z;N_+) \in \K$, is symmetric as Case 2., we conclude \eqref{uhorder1} with $C_v = \max \{C_1, C_2\} $. 
\end{proof}

The semiconcavity property of the discrete value function leads to an improved convergence rate, which we state as the main result of this section below.
\begin{theorem}\label{rate_main}
	Suppose that~\Cref{assp1} and~\Cref{assump_2} hold. 
	There exists a constant $C_1$ depends on $M_f, M_t, L_v, L_f, \upperboundf, \lowerboundf$ such that,  for every $0<h<\frac{1}{\upperboundf}$:
	\begin{equation}\label{rate1_v}
		\|v^h -v\|_\infty \leq C_1 h \ .
	\end{equation}
\end{theorem}
\begin{proof}
	Let us first show that $\sup_{x} (v(x) - v^h(x)) \leq C h$. Since $\A^h \subseteq \A$, we always have 
	\begin{equation}
		\begin{aligned}
			v(x) - v^h(x) &= \inf_{\alpha \in \A} J(x,\alpha) - \inf_{\alpha^h \in \A^h} J^h (x,\alpha^h) \\
			& \leq \inf_{\alpha^h \in \A^h}  J(x,\alpha^h) - \inf_{\alpha^h\in \A^h} J^h(x,\alpha^h) \\
			&\leq \sup_{\alpha^h \in \A^h } ( J(x,\alpha^h) - J^h(x,\alpha^h) ) \ . 
		\end{aligned}
	\end{equation}
	For the discrete time control problem, let us denote $N_x$ such that $y^h_{\alpha^h}(x; N_x)\in \K$ for a given $\alpha^h$. We then have:
	\begin{equation}\label{vvhess}
		\begin{aligned}
			v(x) - v^h(x)  &\leq \sup_{\alpha^h \in \A^h}	(J(x,\alpha^h) - J^h (x,\alpha^h) ) \\
			&\leq \| e^{-\int_{0}^{N_x h } \frac{1}{f(y_{\alpha^h}(x;s),\alpha^h(\lfloor \frac{s}{h} \rfloor)) } ds} - e^{- \sum_{k=0}^{N_x} \frac{h}{f(y^h_{\alpha^h}(x;k),\alpha^h(k))}} \| 
			\leq \frac{L_f}{2 \lowerboundf^2} h \ .
		\end{aligned}
	\end{equation}
	
	To show $\sup_{x}(v^h(x) - v(x)) \leq C h$, we use the same definition of auxiliary functions as the in proof of~\Cref{converg}. Then, similarly as in the proof of~\Cref{converg},
	let $y \to (v^h(x_0) + \theta_{\varepsilon}(x_0-y) + \zeta \xi(x_0,y))$ be the test function, we have:
	\begin{equation}\label{2vh-y0}
		v(y_0) -  ((D\theta_{\varepsilon}(x_0-y_0) \cdot \alpha^{*}- \zeta D_y\xi(x_0,y_0))\cdot \alpha^{*}) f(y_0,\alpha^{*}) - 1 \geq 0 \enspace,
	\end{equation}
	for some $\alpha^{*} \in S_1$. Moreover, let us consider a function
	\begin{equation}
		\vartheta(x) =  v^h(x_0+x) - v^h(x_0) + (D\theta_{\varepsilon}(x_0 - y_0) + \zeta D_x\xi(x_0,y_0))\cdot x \enspace.
	\end{equation}
	Then we have $\vartheta(0) = 0$, and 
	\begin{equation}
		\vartheta(x+z)-2\vartheta(x)+\vartheta(x-z) = v^h(x_0+x+z) - 2v^h(x_0+x)+v^h(x_0+x-z) \leq C_v |z|^2,
	\end{equation}
	by \Cref{semiconcave-value}. Moreover, by the definition of $\psi$, we have 
	\begin{equation} 
		\begin{aligned}
			\vartheta(x) &= \psi(x_0+x,y_0) -\psi(x_0,y_0) + \theta_{\varepsilon}(x_0-y_0) - \theta_{\varepsilon}(x_0+x-y_0)\\
			&\quad+D\theta_{\varepsilon}(x_0-y_0) x + \zeta (\xi(x_0,y_0) - \xi(x_0+x,y_0) + D_x\xi(x_0,y_0)\cdot x) \enspace.
		\end{aligned}
	\end{equation}
	Since $\psi$ gets it's maximum at $(x_0,y_0)$, we then have 
	$
	\limsup_{|x|\to 0}  \frac{\vartheta(x)}{|x|} =  \limsup_{|x| \to 0} (\psi(x_0+x,y_0) - \psi(x_0,y_0)  \leq 0 ,
	$
	which implies
	$
	\vartheta(x) \leq \frac{C_v}{2} |x|^2 .
	$
	
	Let us now take $x = h\alpha^{*}$, then
	\begin{equation}\label{2vh-h-h-1}
		v^h(x_0 + h\alpha^{*}) \leq v^h(x_0) + (D\theta_{\varepsilon}(x_0 - y_0) + \zeta D_x\xi(x_0,y_0))\cdot h\alpha^{*} + \frac{C_v}{2}h^2 \enspace.
	\end{equation}
	Since $v^h$ is the solution of system \eqref{discre-hjb}, we have
	\begin{equation}\label{2vh-h-1}
		v^{h}(x_0) 
		\leq \left\{ (1-\frac{h}{f(x_0, \alpha^{*})}) v^h(x_0 + h \alpha^{*}) + \frac{h}{f(x_0,\alpha^{*})} \right\} \ .
	\end{equation}
	Combining \eqref{2vh-h-h-1} and \eqref{2vh-h-1}, we have
	\begin{equation}
		v^h(x_0) \leq (1-\frac{h}{f(x_0,\alpha^{*})}) (v^h(x_0) + (D\theta_{\varepsilon}(x_0 - y_0) + \zeta D_x\xi(x_0,y_0))\cdot h\alpha^{*} + \frac{C_v}{2} h^2  ) +  \frac{h}{f(x_0,\alpha^{*})} \ ,  
	\end{equation}
	which implies
	\begin{equation}\label{2vh-h-h-3}
		v^h(x_0) \leq (1-\frac{h}{f(x_0,\alpha^{*})})( D\theta_{\varepsilon}(x_0 - y_0) + \zeta D_x\xi(x_0,y_0))\cdot \alpha^{*} + \frac{C_v}{2} h  )  f(x_0,\alpha^{*}) + 1 \enspace.
	\end{equation}
	Combining \eqref{2vh-y0} and \eqref{2vh-h-h-3}, we have:
	\begin{equation}\label{2vh-2}
		\begin{aligned}
			v^h(x_0) - v(y_0) &\leq  (1-\frac{h}{f(x_0,\alpha^{*})})( D\theta_{\varepsilon}(x_0 - y_0) + \zeta D_x\xi(x_0,y_0))\cdot \alpha^{*} + \frac{C_v}{2} h  )f(x_0,\alpha^{*})  \\
			&\quad - ((D\theta_{\varepsilon}(x_0-y_0) \cdot \alpha^{*}- \zeta D_y\xi(x_0,y_0))\cdot \alpha^{*}) f(y_0,\alpha^{*}) \\
			&\leq (f(x_0,\alpha^{*}) - f(y_0,\alpha^{*})) \frac{2|x_0-y_0|}{\varepsilon^2} + \frac{2|x_0-y_0|h}{\varepsilon^2} + \frac{C_v}{2} \upperboundf h + 2 \zeta \upperboundf\\
			&\leq \frac{2L_f|x_0-y_0|^2}{\varepsilon^2} + \frac{2|x_0-y_0|h}{\varepsilon^2} + \frac{C_v}{2} \upperboundf h + 2 \zeta \upperboundf \enspace.
		\end{aligned}
	\end{equation}
	Take $x=y=x_0$ for $\varphi(x,y)$, we obtain 
	$|x_0-y_0| \leq (L_v + \zeta) \varepsilon^2$.
	Thus, for \eqref{2vh-2} we have
	\begin{equation}\label{vhx-vy_1}
		v^h(x_0) - v(y_0) \leq (2L_f (L_v + \zeta)^2 \varepsilon^2 + 2(L_v+\zeta) h + \frac{C_v}{2} \upperboundf h) + 2\upperboundf \zeta \ .
	\end{equation}
	Take now $\varepsilon = h^{\frac{1}{2}}$ in~\eqref{vhx-vy_1}, 
	we then have
	\begin{equation}\label{vhx-vy_2}
		v^h(x_0) - v(y_0) \leq ( 2L_f (L_v + \zeta)^2+ 2(L_v +\zeta) +\frac{C_v}{2} \upperboundf ) h + 2\upperboundf \zeta  \ .
	\end{equation}
	Take $x=y$ for $\psi(x,y)$ and use the fact that $\psi(x_0,y_0) \geq \psi(x,x)$, we have
	\begin{equation}\label{vhx-vx}
		v^h(x) - v(x) \leq (v^h(x_0) -v(y_0)) + \zeta (\xi(x_0,y_0) -\xi(x,x)) -\frac{|x_0-y_0|^{2}}{\varepsilon^2} \ .
	\end{equation}
	Thus, combining~\eqref{vhx-vy_2} and~\eqref{vhx-vx}, we have 
	\begin{equation}\label{essvh}
		v^h(x) - v(x) \leq ( 2L_f (L_v+\zeta)^2+ 2(L_v+\zeta) +\frac{C_v}{2} \upperboundf ) h + (2 \upperboundf -1) \zeta \ .
	\end{equation}
	Taking $\zeta \to 0$ in~\eqref{essvh}, we have
	\begin{equation}\label{vhvess}
		v^h(x) - v(x) \leq (2L_f L_v^2   + 2L_v + \frac{C_v}{2} \upperboundf ) h \enspace.
	\end{equation}
	Combining~\eqref{vvhess} and~\eqref{vhvess}, we conclude that~\eqref{rate1_v} holds with $C_1 = \max\{\frac{L_f}{2 \lowerboundf^2}, 2L_f L_v^2   + 2L_v + \frac{C_v}{2} \upperboundf \}$.
\end{proof}

\section{Convergence of a Full Discretization Scheme, Application to Convergence Rate Analysis of Fast-Marching Method}\label{sec-fulldisfm}

In this section, we first present a full discretization semi-Lagrangian scheme for system~\eqref{HJ}. We demonstrate the convergence rate of the scheme using particular interpolation operators. We then apply this result to show the convergence rate of a fast-marching method, for which the update operator is obtained by applying the presented scheme to eikonal equation.

We shall mention that in previous works on the fast-marching method, in~\cite{sethian2003ordered,cristiani2007fast,carlini2008convergence,Mir14a}, the authors proved the convergence of their methods as the mesh step tends to $0$ without providing an explicit convergence rate, whereas in~\cite{shum2016,mirebeau2019riemannian} the authors establish a convergence rate of order $\frac{1}{2}$ in terms of mesh step. Though, most numerical experiments in the aforementioned works reveal an actual convergence rate of order $1$.

\subsection{A Full Discretization Scheme and a First Convergence Analysis}

To get the numerical approximation of~\eqref{HJ}, we also need to discretize the space. Assume now given a discrete subset $X^h$ of $\R^d$.	
Let us denote $w^h$ the approximate value function for $v$ obtained by applying the semi-Lagrangian scheme~\eqref{discre-hjb} to all grid nodes $x \in X^h$, while when 
the points $x + h \alpha$ are not in the grid $X^h$, we compute the value of $w^{h}(x + h \alpha)$ by an interpolation of the value of it's neighborhood nodes. We assume given an interpolation operator to be used in~\eqref{discre-hjb} when $x \in X^h $. This interpolation may depend on $x$ (this the index of equation), 
and will be denoted by $I^x[\cdot]$. However the value $I^x[w^h](x')$ depends only on the values $w^h(y)$ with $y\in X^h$ in a neighborhood of $x'$. In~\Cref{sec-interpop}, we give a particular piecewise linear operator, defined in a regual mesh grid.

Let us consider the following full discretization semi-Lagrangian scheme, define $z^h : X^h \to \R$ by 
\begin{equation}\label{fully_dis_sl}
	\left\{
	\begin{aligned}
		&z^{h}(x) = \inf_{\alpha \in S_1} \left\{ (1 - \frac{h}{f(x,\alpha)}) I^x[ z^h] (x + h  \alpha)  + \frac{h}{f(x,\alpha)}   \right\}  , &\ x \in X^h \setminus \K \enspace  ,\\
		&z^{h}(x) = 0, & x \in X^h \cap \K \enspace .
	\end{aligned}
	\right.
\end{equation}

Let now $X^h$ be a grid with mesh step $h$. 
For fixed $x \in \R^d$ and $\alpha \in S_1$, let $Y^h(x+ h \alpha) = \{ y_k\}_{k=1\dots d+1}$ denotes a subset of $X^h$. 
We consider particular interpolation operators in~\eqref{fully_dis_sl} that compute the value of $ I^x[ z^h] (x + h  \alpha)$ by convex combination of the values $\{z^h (y_k) \}_{y_k \in Y^h(x+h \alpha)}$. 
Moreover, denote 
\begin{subequations}\label{p1_interpolation}
	\begin{equation}\label{p1_inter}
		I^x[z^h](x+ h \alpha) = \sum_{y_k \in Y^h(x+ h \alpha)}\lambda(x, \alpha,  y_k) z^h(y_k) \ ,
	\end{equation}
	we assume the coefficients $\lambda(x,\alpha, y_k)$ satifify the following condition: 
	\begin{equation}\label{inter_coefficient}
		\left\{
		\begin{aligned}
			& 0 \leq \lambda(x,\alpha, y_k) \leq 1,  \text{ for every } y_k \ , \\
			& \sum_{y_k \in Y^h(x + h \alpha)} \lambda(x, \alpha, y_k) =1 \text{ and }  \sum_{y_k \in Y^h(x + h \alpha)} \lambda(x, \alpha, y_k) y_k = x + h \alpha \ .
		\end{aligned}
		\right.
	\end{equation}
\end{subequations}
Denote $T^h: L^{\infty}(\R^d) \to L^{\infty}(\R^d)$ the operator defined as follows, for every $x \in \R^d$:
\begin{equation}\label{operator_T}
	\left\{
	\begin{aligned}
		&T^h[v^h](x) := \inf_{\alpha \in S_1} \left\{ (1 - \frac{h}{f(x,\alpha)}) v^h(x + h \alpha) + \frac{h}{f(x,\alpha)} \right\}, & x \in \R^d \setminus \K \ , \\
		&T^h[v^h](x) :=0 , & x \in \K \ .
	\end{aligned}
	\right.
\end{equation}
The solution $v^h$ of the time discretization semi-Lagrangian scheme~\eqref{discre-hjb} is indeed a fixed point of $T^h$. 
Similarly let us denote $\overline{T}^h : L^{\infty}(X^h) \to L^{\infty}(X^h)$ the operator associated with the full discretization scheme~\eqref{fully_dis_sl} such that $z^h$ is a fixed point of $\overline{T}^h$. We first state the following technical lemma, which is needed to present our result.
\begin{lemma}[See for instance~\cite{akian2016uniqueness,akian2016minimax}]\label{topnorm}
	Given a $X \subseteq \R^d$, recall that the $\tp-$norm of a function $g: X \to \R$ is defined by 
	\begin{equation}
		\tp(g) = \max\left\{ \sup_{x \in X} \big( g (x) \big), 0\right\} \ . 
	\end{equation} 
	Then both $T^h$ and $\bar{T}^h$ are contraction mappings, in the sense of $\tp-$ norm, in $L^\infty(\R^d)$ and $L^\infty(X^h)$ respectively, with same contracting rate $(1-\frac{h}{\upperboundf})$.
\end{lemma}


In the following, we intend to bound the $L^\infty$ norm between $z^h$ and the restriction of $v^h$ on $X^h$. We begin by showing $z^h-v^h$ in one direction. Let $R_h: \R^{\R^d} \to \R^{X^h}$ be the restriction.
\begin{proposition}\label{rate_wh-v}
Suppose~\Cref{assp1} and~\Cref{assump_2} hold. There exists a constant $C_{z_1}$ depends on $M_f, M_t, \upperboundf, \lowerboundf$ such that, for every $0<h<\frac{1}{\upperboundf}$
\begin{equation}\label{rate_wh1}
	\sup_{x \in X^h} (z^h - \R_h[v^h]) (x) \leq C_{z_1} h \ .
\end{equation}
\end{proposition}
\begin{proof}
By~\Cref{topnorm}, we have
\begin{equation}\label{zh-vh-1}
	\begin{aligned}
		\tp(z^h - R_h [v^h]) 
		&= \tp( z^h - \bar{T}^h \circ R_h [v^h] + \bar{T}^h \circ R_h [v^h] -  R_h [v^h]) \\ 
		& \leq \tp( \bar{T}^h[z^h] - \bar{T}^h \circ R_h [v^h]) + \tp (\bar{T}^h \circ R_h [v^h] -  R_h [v^h]) \\ 
		&\leq (1- \frac{h}{\upperboundf})\tp(z^h - R_h[v^h])   +   \tp (\bar{T}^h \circ R_h [v^h] -  R_h [v^h]) \ .
	\end{aligned}
\end{equation}
Moreover, using the semiconcavity of $v^h$ obtained in~\Cref{semiconcave-value}, for every $x \in X^h$, we have
\begin{equation}\label{zh-vh-2}
	\begin{aligned}
		&(\overline{T}^h \circ R_h[v^h] - R_h[v^h])(x) \\
		&\leq \max_{ \alpha \in S_1} \left\{ (1-\frac{h}{\upperboundf}) \Big( (I^x \circ R_h [v^h])(x + h\alpha) - (R_h[v^h])(x+h \alpha)\Big) \right\} \\ 
		& = \max_{ \alpha \in S_1} \sum_{y_k \in Y^h(x + h \alpha)} \lambda(x,\alpha,y_k)\Big(v^h(y_k) - v^h(x + h \alpha) \Big) \leq C_v h^2 \ .
	\end{aligned}
\end{equation}
Combining~\eqref{zh-vh-1} and~\eqref{zh-vh-2}, we conclude the result in~\eqref{rate_wh1} with $C_{z_1} = \upperboundf C_v$, with $C_v$ the constant obtained in~\Cref{semiconcave-value}.
\end{proof}
\subsection{Controlled Markov Problem and Its Value Function}
In order to show the error bound in the other direction, we will first reformulate the fully discretized equation~\eqref{fully_dis_sl} as a dynamic programming equation of a stochastic optimal control problem. This interpretation has been used in~\cite{kushner2001numerical,bardi2008optimal}. Notice that, in the formulation~\eqref{inter_coefficient}, the coefficients $\{\lambda(x,\alpha, y_k)\}$ can be interpreted as the transition probabilities for a controlled Markov chain, for which the state space is the set of nodes in $X^h$. 
More precisely, let us rewrite the system~\eqref{fully_dis_sl} as follows:
\begin{equation}\label{fully_dis_sl_prob}
\left\{
\begin{aligned}
	&z^{h}(x) = \min_{\alpha \in S_1} \left\{ (1 - \frac{h}{f(x,\alpha)}) \sum_{y_k \in Y^h(x+\alpha h)} \lambda(x, \alpha,y_k) z^h(y_k)    + \frac{h}{f(x,\alpha)} \right\}, & x \in X^h \setminus \K \ ,\\
	&z^{h}(x) = 0, & x \in X^h \cap \K \ .
\end{aligned}
\right.
\end{equation}
Let us consider a Markov decision process $\{\xi_k\}$ on the state space $X^h$, with controls in $S_1$ and transition probability given by
\begin{equation}\label{transition}
\Pro(\xi_{k+1} = y \mid \xi_k = x, \alpha_k = \alpha) = 
\left\{ 
\begin{aligned}
	&\lambda(x,\alpha,y), \  & \text{ if } y \in Y^h(x + \alpha h) \ , \\ 
	& 0, \ & \text{ otherwise } \ .
\end{aligned}
\right.
\end{equation}
Here, $\xi_k$ denotes the state at time step $k$ and $\alpha_k$ denotes the control at time step $k$. Given a pure strategy, that is a map $\sigma^h$ which to any history $H_k = (\xi_0, \alpha_0, \dots, \xi_{k-1},\alpha_{k-1},\xi_k)$ associates a control $\alpha_k$, we can define a probability space $(\Omega, \omega, \Pro)$ and processes $(\xi_k)_{k\geq 0}$ of states and $(\alpha_k)_{k\geq 0}$ of controls satisfying~\eqref{transition}, and $\alpha_k = \sigma^h(H_k)$. 
Let us denote  $N^h(\sigma^h) = \min\{ n \in N_+ \mid \xi_n \in \K \}$, that is a stopping time adapted to the Markov decision process with strategy $\sigma^h$.
By this formulation, we have the following property for the process $(\xi_k)_{k\geq0}$:
\begin{subequations}\label{proper_xi}
\begin{equation}\label{expec_deltaxi}
	\E [ \xi_{k+1} - \xi_{k} \mid \xi_k = x, \alpha_k = \alpha] = h \alpha, \ \forall x\in X^h, \alpha \in S_1 \text{ and } k < N^h(\sigma^h)
\end{equation}
and
\begin{equation}\label{var_deltaxi}
	\Tra( \Var [ \xi_{k+1} - \xi_{k} \mid \xi_k = x, \alpha_k = \alpha]) \leq h^2, \ \forall x\in X^h, \alpha \in S_1 \text{ and } k < N^h(\sigma^h) \ . 
\end{equation}
\end{subequations}	

Let $E^{\sigma^h}_x$ denote the expectation given a initial condition $x$ and a 
strategy $\sigma^h$. 
Consider the following cost functional:
\begin{equation}\label{cost-sto}
W(\sigma^h,x) = \E^{\sigma^h}_x \left[ 1 - \prod_{k=0}^{N^h(\sigma^h)} \left( 1- \frac{h}{f(\xi_k,\alpha_k)} 
\right)  \right] \ .
\end{equation}
Then the solution of~\eqref{fully_dis_sl_prob} is indeed the value function of above controlled Markov problem (see for instance~\cite{kushner2001numerical}), that is, 
\begin{equation}\label{value_wh}
z^h(x) = \inf  \ W(\sigma^h, x) \ .
\end{equation}

\subsection{Convergence Rate Analysis Under A Semiconvexity Assumption}
To get the convergence rate for the full discretization scheme, we shall need the following regularity assumptions for the target set $\K$ and the speed function $f$.
\begin{assumption}\label{assump_3} $ $
\begin{enumerate}
	\item\label{assump_3_1} There exists a constant $M_f'>0$ such that
	\begin{equation}\label{semiconv_f}
		\frac{1}{f(x+z,\alpha)} - 2\frac{1}{f(x,\alpha)} + \frac{1}{f(x-z,\alpha)}  \geq -M_f' |z|^2, \ \forall x,z \in \R^d, \  \forall \alpha \in S_1  \ . \end{equation}
	\item\label{assump_3_2} There exists a constant $M_t'>0$ such that
	\begin{equation}\label{semiconv_dst}
		d_{\K} (x+z) + d_{\K}(x-z) -2 d_{\K}(x)  \geq -M_t' |z|^2, \ \forall x,z \in \R^d  \setminus \mathring{\K}  \ .\end{equation}
\end{enumerate}
\end{assumption}
The assumptions stated in~\eqref{assump_3_1} and~\eqref{assump_3_2} in~\Cref{assump_3} can be thought of as the semiconvexity properties of the speed function and of the distance function $d_\K$ in $\R^d \setminus \mathring{\K}$, respectively. In particular, if $f$ is of class $\Cc^2$ and $\partial \K$ is of class $\Cc^2$, one can check that both~\Cref{assump_2} and~\Cref{assump_3} hold.

We first state the following technical lemma, which is needed to present our result. 
\begin{lemma}\label{jensen_semiconvex}
Assume $g : \R^d \to \R$ is $\tilde{\alpha}$-semiconvex for $\tilde{\alpha} \geq 0$, then we have 
\begin{equation}\label{bound_gE-Eg}
	g(E[X]) - E[g(X)] \leq \frac{\tilde{\alpha}}{2} \Tra(\Var[X]) \ .
\end{equation}
\end{lemma} 
\begin{proof}
Since $g(x)$ is $\tilde{\alpha}-$semiconvex, we have $g(x) + \frac{\tilde{\alpha}}{2} \|x\|^2$ is convex, then
\begin{equation}
	\E[ g(X) + \frac{\tilde{\alpha}}{2} \|X\|^2 ] \geq g(\E[X])+ \frac{\tilde{\alpha}}{2} \|\E[X]\|^2] \ .
\end{equation}
\eqref{bound_gE-Eg} is then deduced using $E[X^2] - (E[X])^2 = \Var[X]$.
\end{proof}

\begin{corollary}\label{jensen_semiconcave}
Assume $g: \R^d \to \R$ is $\tilde{\alpha}-$semiconcave for $\tilde{\alpha} \geq 0$, then we have
\begin{equation}
	E[g(X)] - g(E[X])  \leq \frac{\tilde{\alpha}}{2} \Tra(\Var[X]) \ .
\end{equation}
\end{corollary}

\begin{proposition}\label{rate_v-wh}
Suppose~\Cref{assp1} and~\Cref{assump_3} hold. There exists a constant $C_{z_2}$ depending on $M_f', M_t', \upperboundf, \lowerboundf$ such that, for every $0<h<\frac{1}{\upperboundf}$
\begin{equation}\label{rate_wh2}
	\sup_{x \in X^h} (v^h - z^h)(x) \leq C_{z_2} h \ .
\end{equation}
\end{proposition}
\begin{proof}
Let us denote $\sigma^{h}$ a strategy for the stochastic control problem. Let $(\Omega, \mathcal{F}, \Pro)$, $(\xi_k)$, $(\alpha_k)$ and $N^h(\sigma^h)$ be defined as above. 
Now, when $\omega \in \Omega$ is fixed together with $(\alpha_k(\omega))_{0\leq k \leq N^h(\sigma^h)}$ the associated control, consider 
a deterministic trajectory $\{\bar{\xi}_k\}_{k = 1,2,\dots}$ such that
\begin{equation}\label{deter-tra}
	\bar{\xi}_0 = x, \quad \bar{\xi}_{k+1} = \bar{\xi}_k + h\alpha_k(\omega) \ ,
\end{equation}
and if $\bar{\xi}_{n} \notin \K$ for all $n \leq N^h(\sigma^h)$, we then take the controls following the straight line 
from $\bar{\xi}_{N^h(\sigma^h)}$ to $\K$. Let us denote $\bar{N}^h(\omega, \sigma^h) = \min \{ n \in \N_+ \mid \bar{\xi}_n \in \K  \}$. 
By this construction, $\{\bar{\xi}_{k}\}_{0 \leq k \leq \bar{N}^h(w,\sigma^h)}$ is indeed a solution of the discrete system~\eqref{discre_dyn}, i.e., $\bar{\xi}_k$ satisfies $\bar{\xi}_k = y^h_{\alpha^h}(x,k)$ for every $k\in \{0,1\dots, \bar{N}^h(\omega, \sigma^h)\}$. Thus, we have
\begin{equation}
	v^h(x) \leq  J^h ( \alpha^h, x) \ . 
\end{equation}
Since this holds for almost all $\omega \in\Omega$, we have 
\begin{equation}
	v^h(x) \leq \E^{\sigma^h}_x[ J^h(\alpha^h,x) ] \ .
\end{equation}
Let us simply denote $N^h(\sigma^h)$ by $N$ and $\bar{N}^h(\omega,\sigma^h)$ by $\bar{N}$ in the following. We have 
\begin{equation}\label{vh-wh}
	\begin{aligned}
		v^h(x) - z^h(x) &\leq \E^{\sigma^h}_x \left[ \left(1 - \prod_{k=0}^{\bar{N}} \Big(1 - \frac{h}{f(\bar{\xi}_k,\alpha_k)} \Big) \right) - \left(  1 - \prod_{k=0}^{N}\Big(1 - \frac{h}{f(\xi_k,\alpha_k)} \Big) \right) \right] \\
		&\leq \E^{\sigma^h}_x \left[ \mathbbm{1}_{\bar{N} \leq N} \left\{ \prod_{k=0}^{\bar{N}} \Big(1 - \frac{h}{f(\xi_k,\alpha_k)}  \Big) - \prod_{k=0}^{\bar{N}} \Big(1 - \frac{h}{f(\bar{\xi}_k,\alpha_k)} \Big) \right\}  \right. \\
		& \qquad \ + \left. \mathbbm{1}_{N<\bar{N}} \left\{ \prod_{k=0}^N\Big( 1-\frac{h}{f(\xi_k,\alpha_k)} \Big) - \prod_{k=0}^N \Big(1 - \frac{h}{f(\bar{\xi}_k,\alpha_k)} \Big) \right. \right. \\
		&\qquad \qquad  + \left. \left. \left( \prod_{k=0}^N\Big(1 - \frac{h}{f(\bar{\xi}_k,\alpha_k)} \Big) \right) \left(1- \prod_{k=N+1}^{\bar{N}}\Big(1 - \frac{h}{f(\bar{\xi}_k,\alpha_k)} \Big)  \right) \right\} \right] \ .
	\end{aligned}
\end{equation}
First notice that 
\begin{equation}\label{bound_dif}
	\begin{aligned}
		&\E^{\sigma^h}_x \left[ \frac{1}{f(\bar{\xi}_k,\alpha_k)} -\frac{1}{{f(\xi_k,\alpha_k)}} \mid k \leq (N \land \bar{N})  \right] \\ 
		&= \E^{\sigma^h}_x \left[\E \Big[ \frac{1}{f(\bar{\xi}_{k-1} + h\alpha_{k-1},\alpha_k)} - \frac{1}{{f(\xi_{k-1} + \xi_k - \xi_{k-1},\alpha_k)}} \right. \\
		&\qquad \qquad \quad \left.\mid \xi_{k-1}, k-1\leq (N \land \bar{N}) \Big] \mid k \leq (N \land \bar{N})   \right]  \\ 
		& \leq \E^{\sigma^h}_x \left[ \frac{1}{f(\bar{\xi}_{k-1} + h\alpha_{k-1},\alpha_k)} - \frac{1}{f( \xi_{k-1}+h\alpha_{k-1},\alpha_k )} \mid k \leq (N \land \bar{N}) \right]  \\
		&\quad+ M_f' (\Tra(\Var [\xi_k - \xi_{k-1}]))  \ ,
	\end{aligned}
\end{equation}
where the last inequality is deduced by~\eqref{expec_deltaxi} and~\Cref{jensen_semiconvex}. 
Then, by induction and by~\eqref{var_deltaxi}, we have
\begin{equation}\label{bound_dif2}
	\E^{\sigma^h}_x \left[ \frac{1}{f(\bar{\xi}_k,\alpha_k)} -\frac{1}{{f(\xi_k,\alpha_k)}} \mid k \leq (N \land \bar{N})  \right] \leq k M_f' h^2 \ .
\end{equation}
Let us now focus on the first part in the summation of~\eqref{vh-wh}, for which we have
\begin{equation}\label{vh-wh-1}
	\begin{aligned}
		&\E^{\sigma^h}_x \left[ \mathbbm{1}_{\bar{N} \leq N} \left\{ \prod_{k=0}^{\bar{N}} \Big(1 - \frac{h}{f(\xi_k,\alpha_k)}  \Big) - \prod_{k=0}^{\bar{N}} \Big(1 - \frac{h}{f(\bar{\xi}_k,\alpha_k)} \Big) \right\}  \right] \\ 
		&\leq \E^{\sigma^h}_x \left[  \mathbbm{1}_{\bar{N} \leq N} \left\{ \Phi_k(\xi_k) \sum_{k=1}^{\bar{N}} \Big(\frac{h}{f(\bar{\xi}_k, \alpha_k)} - \frac{h}{f(\xi_k,\alpha_k)} \Big) \right\} \right] \\
		&\leq h \E^{\sigma^h}_x \left[ \mathbbm{1}_{\bar{N} \leq N} (1 - \frac{h}{\upperboundf})^{\bar{N}} \sum_{k=1}^{\bar{N}} k M_f' h^2  \right] \\
		& \leq \Pro(\mathbbm{1}_{\bar{N} \leq N}) 2 M_f' \upperboundf^2 h \ ,
	\end{aligned}
\end{equation}
where $\Phi_k(\xi_k)$ is a random variable and $\Phi_k(\xi_k) \leq (1 - \frac{h}{\upperboundf})^{\bar{N}}$. 
For the second part in~\eqref{vh-wh}, the first part of the sum is bounded by the same form as computed in~\eqref{vh-wh-1}. 
As for the remaining part, we notice that, by a similar computation as in~\eqref{bound_dif}:
\begin{equation}\label{bound_did}
	\E^{\sigma^h}_x \left[ d_\K(\bar{\xi}_k) - d_\K(\xi_k) \mid k \leq (N\land \bar{N}) \right] \leq k M_t' h^2 . 
\end{equation}
Then, we have 
\begin{equation}\label{vh-wh-2}
	\begin{aligned}
		&\E^{\sigma^h}_x \left[ \mathbbm{1}_{N < \bar{N}} \left\{ \left(\prod_{k=0}^N\Big(1 - \frac{h}{f(\bar{\xi}_k,\alpha_k)} \Big)\right)\left(1- \prod_{k=N+1}^{\bar{N}}\Big(1 - \frac{h}{f(\bar{\xi}_k,\alpha_k)} \Big)  \right) \right\} \right] \\
		&\leq \E^{\sigma^h}_x \left[ \mathbbm{1}_{N < \bar{N}} \left\{(1 - \frac{h}{\upperboundf})^N \sum_{k=N}^{\bar{N}} \Big(\frac{h}{f(\bar{\xi}_k,\alpha_k)} \Big) \right\} \right] \\ 
		& \leq \E^{\sigma^h}_x \left[ \mathbbm{1}_{N < \bar{N}}  (1 - \frac{h}{\upperboundf})^N\frac{1}{\lowerboundf} \Big( d_\K(\bar{\xi}_N) - d_\K(\xi_N) \Big) \right] \\
		& \leq  \E^{\sigma^h}_x \left[ \mathbbm{1}_{N < \bar{N}} (1 - \frac{h}{\upperboundf})^N\frac{1}{\lowerboundf} N M_t' h^2  \right] \\
		& \leq \Pro( \mathbbm{1}_{N < \bar{N}} ) \frac{M'_t}{\lowerboundf} h \ . 
	\end{aligned}
\end{equation}
Combing~\eqref{vh-wh-1} and~\eqref{vh-wh-2}, we have
\begin{equation}
	v^h - z^h \leq \Pro(\mathbbm{1}_{\bar{N}\leq N}) 2M'_f \upperboundf^2 h + \Pro(\mathbbm{1}_{N<\bar{N}}) (2M'_f \upperboundf^2 + \frac{M'_t}{\lowerboundf}) h \leq (2M'_f \upperboundf^2 + \frac{M'_t}{\lowerboundf}) h \ .
\end{equation}
The result in~\eqref{rate_wh2} is then concluded with $C_{z_2} = 2M'_f \upperboundf^2 + \frac{M'_t}{\lowerboundf}$.
\end{proof}

\begin{corollary}\label{coro_zh-v}
Suppose~~\Cref{assp1}, \Cref{assump_2} and~\Cref{assump_3} hold. Take $C_z = \max \{C_{z_1}, C_{z_2}\}$, where $C_{z_1}$ satisfies the condition in~\Cref{rate_wh-v} and $C_{z_2}$ satisfies the condition in~\Cref{rate_v-wh}. For every $0 < h < \frac{1}{\upperboundf}$, we have 
\begin{equation}\label{zh-vh-sup}
	\sup_{x\in X^h} \|z^h(x) -v^h(x)\| \leq C_z h \ .
\end{equation}
\end{corollary}

The following theorem is then a  direct consequence of~\Cref{rate_main} and~\Cref{coro_zh-v}, which we state as the main result of this subsection.
\begin{theorem}\label{convergence_rate_1}
Suppose that \Cref{assp1},~\Cref{assump_2} and ~\Cref{assump_3} hold. There
exists a constant $C_1'$ depends on $L_f,L_v, M_f, M_t, M_f' , M_t', \upperboundf, \lowerboundf$, such that, for every $0<h<\frac{1}{\upperboundf}$:
\begin{equation}\label{rate_1}
	\sup_{x \in X^h}\|z^h (x) - v(x) \| \leq C_1' h \ .
\end{equation}
\end{theorem}

\subsection{Application: Convergence and Complexity of Fast Marching Method}
In this subsection, we will apply our results in the full discretization scheme to analyze the convergence rate and computational complexity of fast marching methods, for which the update operators are derived from a semi-Lagrangian type discretization of the eikonal equation. 
\subsubsection{A Particular Piecewise Linear Interpolation Operator}\label{sec-interpop}
In this section, we will present a specific piecewise linear interpolation operator, 
for the full discretization semi-Lagrangian scheme, 
that leads to an efficient implementation, 
particularly for the isotropic eikonal equation. Notice that computing the minimum in~\eqref{fully_dis_sl} is not trivial, especially when the dimension is high. Moreover, generally, in the $d$ dimensional case, we need at least the value in $d+1$ nodes of the grid, in order to compute the interpolation in one node.
We describe here one possible way to define an interpolation operator and 
to compute the minimum in $\eqref{fully_dis_sl}$, within a regular grid with  space mesh step equal to $h$ in every dimension, i.e.,  $\Delta x_i  = h, \forall i \in \{1,2,\dots,d \}$. This interpolation operator is based on the work of~\cite{cristiani2007fast},
in which the convergence is shown in the isotropic case.

Let $x = (x_{1},x_{2},\dots,x_{d})$ denote a point of $X$.
Roughly speaking, 
the $d-$dimen\-sional space 
is ``partitioned'' into $2^d$ orthants. 
We consider only the open orthants, since their boundaries are negligible. Let us denote
by $V$ the approximate value function
in the grid point $x\in X$. 
The values of the interpolation $I^x[V](x+h\alpha)$
with $\alpha\in S_1$ 
are defined (differently) for $\alpha$ in each orthant, and the minimum value in each orthant is first computed. Then, the minimum will be obtained by further taking the minimum among the values in all orthants. 

Denote by $e_1,\ldots, e_d$ the vectors of the canonical basis of $\R^d$.
We compute the minimum in the positive orthant using $d+1$ nodes: $x^l:=x + he_l, l\in \{1,\dots,d\} $, 
and $x_{+1}:=x+h (e_1+e_2+\dots+ e_d)$. The minimum in other orthants will be computed using the same method. 

The interpolated value function in $x+h\alpha$
with $\alpha$ in the positive orthant of the sphere $S_1$, denoted by $v^{s,1}$, will be given by the linear interpolation of $V(x^1)$, $V(x^2)$,
\dots, $V(x^d)$ and $V(x_{+1})$, which is equal to
\begin{equation}\label{interp_shpere2}
\begin{aligned}
	v^{s,1}(x+h\alpha) 
	&= \sum_{k=1}^d \alpha_k V(x^k)
	+\frac{V(x_{+1})-\sum_{l = 1}^{d}V(x^l)}{d-1}
	\Big( (\sum_{\ell=1}^d \alpha_\ell ) -1\Big)\enspace .
\end{aligned}
\end{equation}
We then use $(\theta_{1},\theta_2,\dots,\theta_{d-1}$), $\theta_k \in (0,\frac{\pi}{2})$, to represent a vector $\alpha\in S_1$ belonging to the positive orthant, that is 
\begin{equation}\label{alpha-theta2}
\alpha_1 =  \cos(\theta_1),     \alpha_2 =  \sin(\theta_1)\cos(\theta_2),  
\dots,  \alpha_d = \sin(\theta_1) \sin(\theta_2)\cdots \sin(\theta_{d-1}) \ .
\end{equation}
This allows one to rewrite \eqref{interp_shpere2} as
a function of $(\theta_1,\theta_2,\dots,\theta_{d-1})$. 
By doing so, one can consider the result of the optimization in
the first equation of \eqref{fully_dis_sl}, with $w^h$ replaced by $V$ and $I$ replaced by $I^x$, restricted to the
positive orthant, as an approximate value of $V(x)$, denoted by $V^1$, and given by:
\begin{equation}\label{fmoperator_s_22}
V^{1}(x)
= \min_{\theta_1,\dots,\theta_{d-1}} \left\{ (1-\frac{h}{f(x,\alpha)})v^{s,1} (x+h \alpha)
+ \frac{h}{f(x,\alpha)} \right\} \enspace.
\end{equation}
Notice that the minimum in equation \eqref{fmoperator_s_22} is easier to compute by taking the minimum first on $\theta_{d-1}$, then $\theta_{d-2}$, until $\theta_{1}$. Indeed,  we notice in \eqref{alpha-theta2}, that only the last two entries of $\alpha$ contain $\theta_{d-1}$. Thus, the minimum of \eqref{fmoperator_s_22} over $\theta_{d-1}$ can be computed separately.
Moreover, in the isotropic case, meaning $f(x,\alpha) \equiv f(x), \forall \alpha \in S_1$, the minimal $\theta_{d-1}$ is independent of
$\theta_1,\ldots \theta_{d-2}$, due to the special form of \eqref{alpha-theta2}
and \eqref{interp_shpere2}.
The iteratively computation over $\theta_{d-2}$ to $\theta_1$ will be the same.

Then, the full discretization scheme, using the interpolation operator described as above, is as follows:
\begin{equation}\label{fully_dis_sl_2}
\left\{
\begin{aligned}
	&V(x_i) = \min_{k \in \{1,2,\dots,2^d\}} V^k(x_i), \quad &x_i \in X^h \cap (\R^d \setminus \K) \ ,  \\ 
	&V(x_i)= 0, & x_i \in X^h \cap \K \enspace .
\end{aligned}
\right.
\end{equation}
\begin{proposition}[Corollary of~\Cref{convergence_rate_1}]\label{propo_spop}
Suppose that \Cref{assp1},~\Cref{assump_2} and ~\Cref{assump_3} hold. Take $I^x$ as in~\eqref{fully_dis_sl_2}. There exists a constant $C_V$ depends on $L_f,L_v, M_f, M_t, M_f', M_t', \upperboundf$, $\lowerboundf$, such that, for every $0<h<\frac{1}{\upperboundf}$:
\begin{equation}
	\sup_{x \in X^h} \| V(x) - v(x) \| \leq C_V h \ .
\end{equation}
\end{proposition}

\subsubsection{The Fast-Marching Method and Its Convergence}\label{sec-fm}
We briefly recall the fast marching method introduced by Sethian~\cite{sethian1996fast} and Tsitsiklis~\cite{tsitsiklis1995efficient}, which is one of the most effective numerical methods to solve the eikonal equation. 
Its initial idea takes advantage of the property that the evolution of the domain encircled by the front is monotone non-decreasing, thus one is allowed to only focus on the computation around the front at each iteration.
Generally, it has computational complexity (number of arithmetic operations) in the order of $K_d M \log (M)$  in a $d$-dimensional grid with
$M$ points 
(see for instance~\cite{sethian1996fast,cristiani2007fast}), 
where the constant $K_d$ depends on the discretization scheme.

The fast marching method is searching the nodes of $X$ according to a special ordering and computes the approximate value function in just one iteration. The special ordering 
is constructed in such a way that the value function is monotone non-decreasing in the direction of propagation. This construction is done by dividing the nodes into three groups (see below figure): \Far, which contains the nodes that have not been searched yet; \Accepted, which contains the nodes
at which
the value function has been already computed and settled -- by the monotone property, in the subsequent search, we do not need to update the value function of those nodes; and \Narrow, which contains the nodes "around" the front -- at each step, the value function is updated only at these nodes.
\begin{figure}[H]
\centering
\includegraphics[width=0.7\textwidth]{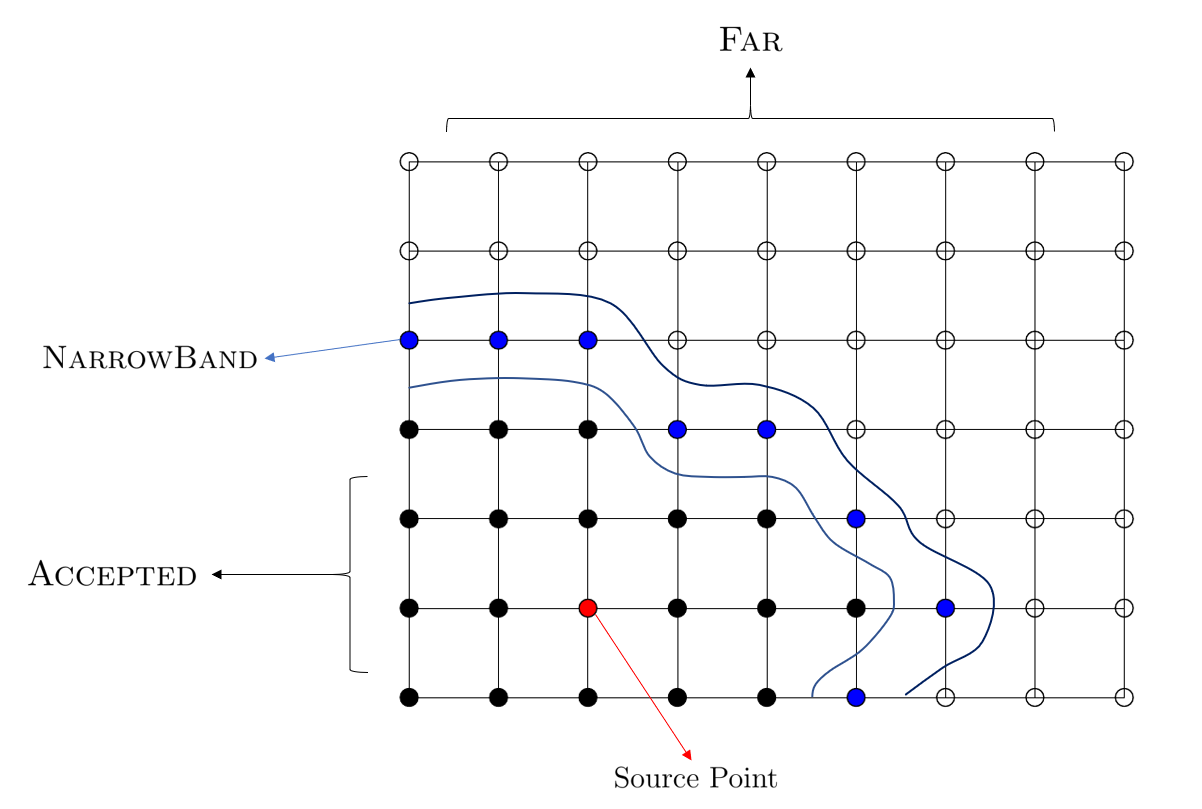}
\label{fmfigure}
\end{figure}
At each step, the node in \Narrow  with the smallest value is added into
the set of \Accepted nodes, and then the \Narrow and the value function over \Narrow  are updated, using the value of the last accepted node. The computation is done by appying an update operator $\mathcal{U}: (\R \cup \{+\infty\})^{X} \to (\R \cup \{+\infty\})^{X}$, which is based
on the discretization scheme.
The classical update operators are based on finite-difference (see for instance \cite{sethian1996fast}) or semi-lagrangian discretizations (see for instance \cite{cristiani2007fast}).
Sufficient conditions on the update operator $\mathcal{U}$ 
for the convergence of the fast marching algorithm 
are that the approximate value
function on $X$ is the unique fixed point of $\mathcal{U}$
satisfying the boundary conditions,
and that  $\mathcal{U}$ is monotone and causal \cite{sethian1996fast}.

A generic partial fast marching algorithm is given in \Cref{fmalgo}.

\begin{algorithm}[htbp]
\caption{Fast Marching Method (compare with \cite{sethian1996fast,cristiani2007fast,mirebeau2019riemannian}). } 
\label{fmalgo}
\hspace*{0.02in} {{\bf Input:} Mesh grid $X$; Update operator $\mathcal{U}$. Set of nodes in target set \Start. }  \\
\hspace*{0.02in} {{\bf Output:} Approximate value function $V$ and \Accepted set.}  \\
\hspace*{0.02in} {{\bf Initialization:} Set $V(x) = +\infty,\forall x \in X$. Set all nodes as \Far. }
\begin{algorithmic}[1]
\State Add \Start to \Accepted, add all neighborhood nodes to \Narrow. 
\State Compute the initial value $V(x)$ of the nodes in \Narrow.
\While {(\Narrow \ is not empty)}
\State Select $x^{*}$ having the minimum value $V(x^{*})$ among the \Narrow \ nodes.
\State Move $x^{*}$ from \Narrow \ to \Accepted.  
\For{All nodes $y$ not in \Accepted, such that $\mathcal{U}(V)(y)$ depends on $x^{*}$}
\State $V(y) = \mathcal{U} (V)(y)$
\If y is not in \Narrow
\State Move $y$ from \Far \ to \Narrow.
\EndIf
\EndFor
\EndWhile
\end{algorithmic}
\end{algorithm}

Let us now consider the fast-marching method with a particular update operator as described in~\eqref{fully_dis_sl_2}. I.e., we set the input update operator in~\Cref{fmalgo} as follows:
\begin{equation}\label{update_fm}
\mathcal{U}(V)(x) := \min_{k\in\{1,2,\dots,2^d\}} V^k(x) \ ,
\end{equation}
where $V^k$ is defined similarly as in~\eqref{fmoperator_s_22}. Then, we have the following result:
\begin{proposition}\label{rate_fm}$ $
\begin{enumerate}
\item\label{rate_fm_1} Suppose~\Cref{assp1},~\Cref{assump_2} and~\Cref{assump_3} hold. We have
\begin{equation}\label{converg_fm}
	\sup_{x \in X} \|V(x) - v(x) \| \leq C_V h \ ,
\end{equation}
for every $h \leq \frac{1}{\upperboundf}$. 
\item\label{rate_fm_2} 	In order to get an error between $V$ and the value function of the problem~\eqref{eikonal_eq} less of equal $\varepsilon$, we shall take the mesh grid $h = (C_V)^{-1} \varepsilon$. Then, the total computational complexity of the fast-marching method is $\tilde{O}\Big((2 C_V \varepsilon^{-1})^{d}\Big)$.
\end{enumerate}
\end{proposition}
\begin{skproof}
\eqref{rate_fm_1} is a direct result as $\mathcal{U}$ satisfies the fast marching technique (see for instance~\cite{cristiani2007fast}), and the result in~\Cref{propo_spop}. 

For~\eqref{rate_fm_2}, the choice of the mesh step $h$ is derived from the error estimates in~\eqref{converg_fm}. This results in the presence of $O((C_v \varepsilon^{-1})^d)$ nodes in the grid.
Moreover, one step update using the update operator~\eqref{update_fm} needs $O(d \times 2^d)$ arithmetic operations, and the fast-marching method needs a number of update steps in the order of $O(M\log(M))$ to operate on a grid with $M$ nodes. The result is then concluded.
\end{skproof}

\section{Convergence Under State Constraints, Application to Computational Complexity of The Multilevel Fast-Marching Method}\label{sec-mlfm}

In~\Cref{sec-semilagrangian} and~\Cref{sec-fulldisfm}, we presented convergence results when the state space  is $\R^d$. However, in the context of optimal control problems, it is natural to consider that the state is required to stay within the closure $\Ocb$ of a certain open domain $\Oc$ in $\R^d$, which is also often the case in practical applications. 
This is in particular relevant for numerical approximations, even in unconstrained problems. For computation and implementation, one need to discretize both time and space, requiring the grid to be generated within a specific domain of $\R^d$.

Another goal of this section is to give a sufficient condition for the computational complexity analysis of the multi-level fast marching method of~\cite{akian2023multi}, in which a particular state constraint is introduced, satisfying that the $\delta-$geodesic set between two target sets remains within such a state constraint. 

\subsection{HJB Equation for State Constrained Problems}
Let us first briefly recall the framework of state constrained optimal control problems and HJB equation. For the control problem~{\rm (\ref{dynmsys},\ref{cost},\ref{value})} presented in~\Cref{sec-pre}, we further require that the state $y$ stays inside the domain $\Ocb$,  by considering the following set of controls:
\begin{equation}
\A_{\Oc,x}: = \{ \alpha \in \A \mid y_\alpha(x;s) \in \Ocb, \ \text{for all } s\geq 0 \ \} \ .
\end{equation}
Under~\Cref{assp1}, $\A_{\Oc,x} \neq \emptyset$. We denote $v_{\Oc}$ the value function of the state constrained control problem, that is 
\begin{equation}\label{value_stateconstraint}
v_\Oc(x) = \inf_{\alpha \in \A_{\Oc,x}} J(\alpha(\cdot),x) \ .
\end{equation}
Assume further $\partial \Oc \setminus \partial \K$ is of class $\Cc^1$, $v_\Oc$ is the unique viscosity solution of the following stated constrained HJB equation (see for instance~\cite{soner1986optimal1}, \cite{capuzzo-lions}, and in particular~\cite{akian2023multi} for the same set of  problem):
\begin{equation}\label{HJ_constrained}
\left\{
\begin{aligned}
& F(x, v_\Oc(x), Dv_\Oc(x)) = 0, \enspace & x \in \Oc \  , \\
& F(x, v_\Oc(x), Dv_\Oc(x)) \geq 0, \enspace & x \in \partial \Oc \setminus \partial \K \  , \\
&v_\Oc(x) = 0, \enspace & x \in \partial \K \ .
\end{aligned}
\right.
\end{equation}

\subsection{Convergence of the time discretization scheme}
We first consider a semi-Lagrangian scheme involves only time discretization for the system~\eqref{HJ_constrained}, which is as follows
\begin{equation}\label{discre-hjb-constraint}
\left\{
\begin{aligned}
&v_\Oc^h(x) = \min_{\alpha \in S_{1}} \left\{ \big(1-\frac{h}{f(x,\alpha)}\big) v_\Oc^h(x + h \alpha)  + \frac{h}{f(x,\alpha)}   \right\} , \ & x \in \Ocb \setminus \K  \ ,\\
& v_\Oc^h(x) = 1, & x \notin \Ocb \ ,\\
&v_\Oc^h(x) = 0, \ &x \in  \K \ .
\end{aligned}
\right.
\end{equation} 
Notice that the equation~\eqref{discre-hjb-constraint} also has an interpretation of discrete time optimal control problem, as we formulated in~\Cref{subsec_discre_control}, with state constraint $\Oc$, that is when $x \notin \Ocb$, the cost is $1$ instead of $\frac{h}{f(x,\alpha)}$. 

For the continuous time unconstrained problem~{\rm (\ref{dynmsys},\ref{cost},\ref{value})}, let us define the $\delta-$geodesic set from a point $x \in \R^d \setminus \K$ to $\K$ as follows: for a give $\alpha^\delta$, denote $\tau^\delta$ such that $y_{\alpha^\delta}(x;\tau^\delta) \in \K $:
\begin{equation}\label{delta-geodesic}
\Gamma^\delta_x := \{ y_{\alpha^\delta}(x;s) \mid s \in [0,\tau^\delta], \  J(x,\alpha^\delta(\cdot)) \leq v(x) + \delta \ \} \ .
\end{equation}
Moreover, for the discrete time unconstrained control problem~{\rm (\ref{discre_dyn},\ref{dis_step},\ref{dis_cost},\ref{value_discret})} with initial state $x$, let us denote by $(\alpha^{h,*}_0,\dots, \alpha^{h,*}_{N_x-1} )$ the sequence of discrete optimal controls. We define the interpolation of the discrete optimal trajectory as the following set:
\begin{equation}\label{discret_tra}
\Gamma^{h,*}_x : = \{ y_{\alpha}(x;s) \mid s \in [0,N_x h], \ \alpha(s) = \alpha^{h,*}_{\lfloor \frac{s}{h} \rfloor } \} \ .
\end{equation}  
Then, by~\Cref{rate_main} and in particular the computation in~\eqref{vvhess}, we have that $\Gamma^{h,*}_x$ is included in the $C_1 h-$geodesic set for the continuous problem starting from $x$, that is $\Gamma^{h,*}_x \subseteq \Gamma^{C_1 h}_x$. 
Let us define a subset $\Oc_\delta$ of $\Oc$ as follows: 
\begin{equation}\label{def_Oh}
\Oc_\delta := \{x\in \Oc \mid \Gamma^\delta_x \subset \Oc \   \} \ .
\end{equation}

Then, 
we have the following results:
\begin{theorem}\label{propo-constraint}
Suppose~\Cref{assp1} and~\Cref{assump_2}  
hold. 
Assume that $\Oc_\delta$ is nonempty and contains $\K$.
For every $x \in \Oc_\delta$ and for every $h < \min\{ \frac{\delta}{C_1}, \frac{1}{\upperboundf} \}$, we have $v(x) = v_\Oc(x)$ and $v^h(x) = v^h_\Oc(x)$. Thus, 
\begin{equation}\label{vhconstraint}
|v^h_\Oc(x) - v_\Oc(x) | \leq C_1 h \ ,
\end{equation}
where $C_1$ is the constant obtained in~\Cref{rate_main}. 
\end{theorem} 
\begin{proof}
We first show that for every $x \in \Oc_\delta$, $v(x) = v_\Oc(x)$, see also~\cite[Theorem 3.13]{akian2023multi}. Indeed, $v(x) \leq v_\Oc(x)$ is automatic, since $\Oc \subset \R^d$. For the reverse inequality, it is sufficient to consider an $\epsilon-$optimal trajectory such that $\epsilon < \delta$ thus it is inside $\Oc_\delta$. Then we have $v_\Oc(x) \leq v(x) + \epsilon$. Since this holds for arbitrary small $\epsilon$, we deduce $v_\Oc(x) \leq v(x)$ and so the equality. 

We next show that for every $x \in \Oc_\delta$, $ v^h(x) = v^h_\Oc(x)$. The inequality $v^h(x) \leq v^h_\Oc(x)$ is again automatic since $\Oc \subset \R^d$. For the reverse direction, 
since $C_1h \leq \delta$, we have $\Gamma^{h,*}_x \subseteq \Gamma^\delta_x \subseteq \Oc$. By a similar argument as above, we deduce $v^h_\Oc(x) \leq v^h(x)$ and so the equality. 

The result in~\eqref{vhconstraint} follows from~\Cref{rate_main}, as for $x\in \Oc_\delta$, $| v^h_\Oc(x) - v_\Oc(x) | = |v^h(x) - v(x) |$. 
\end{proof}

\subsection{Convergence of the full discretization scheme}
Similarly as in~\Cref{sec-fulldisfm}, we consider a full discretization semi-Lagrangian scheme which involves applying an interpolation operator $I^x$, that is 
\begin{equation}\label{fully_dis_sl_con}
\left\{
\begin{aligned}
&z_\Oc^h(x) = \min_{\alpha \in S_1} \left\{ (1 - \frac{h}{f(x,\alpha)}) I^x[ z_\Oc^h] (x + h  \alpha)  + \frac{h}{f(x,\alpha)}   \right\}  , &\ x \in X^h \cap (\Ocb \setminus \K) \ ,\\
& z_\Oc^h(x) = 1, & x \notin X^h \cap \Ocb \ , \\
&z^{h}(x) = 0, & x \in X^h \cap \K \enspace .
\end{aligned}
\right.
\end{equation}
In particular, the interpolation operator $I^x$ we are interested in here satisfies the same property as in the unconstrained case, that is the property described in~\eqref{p1_interpolation}. 

In order to show the convergence rate of $z_\Oc^h$ to $v_\Oc$, we shall need a further assumption on the boundary of the state constraint. Let us denote 
\begin{equation}
\dd_\Oc(x) = \inf_{y \in \partial \Oc} \|y -x \| , \  \text{ for every } x \in \R^d \ ,
\end{equation} 
the distance function from $x$ to the boundary of $\Oc$. 
\begin{assumption}\label{assump_5} 
There exists a constant $M_\Oc >0$ such that, for every $x,z\in \overline{\Oc} \setminus \K$: 
\begin{equation}
\dd_\Oc(x+z) + \dd_\Oc(x-z) -2 \dd_\Oc(x) \geq -M_\Oc \|z \|^2 \ .
\end{equation}
\end{assumption}
\Cref{assump_5} can be thought of as a semiconvexity condition on the distance function to the boundary of $\Oc$, in addition to the $\Cc^1$ assumption to get the uniqueness of viscosity solution. 
Moreover, we state the following lemmas which are needed to present our result:
\begin{lemma}\label{lemma_distancetobound}
Let $\delta>0$ and 
$h < \frac{\delta}{C_1}$. For $y \in \Oc_{2 \delta}$,
we have for every $x \in \Gamma^{h,*}_y$, 
\begin{equation}
\dd_\Oc(x) \geq \frac{\lowerboundf}{2} \delta \ .
\end{equation}
\end{lemma}
\begin{proof}
Take  $y \in \Oc_{2 \delta}$ and $x \in \Gamma^{h,*}_y \subseteq \Gamma^{\delta}_y$, 
where the last inclusion is deduced from $C_1 h < \delta$.  Thus,
there exists a $\delta$-optimal trajectory $y^1(s)$ from $y^1(0)=y$ to $y^1(\tau_\K) \in \K$
such that $y^1(t)=x$ for some $0\leq t\leq \tau_\K$.
For any $\delta$-optimal trajectory from $x$ to $\K$, $y^2(s)$ from $y$ to $K$,
the trajectory $y(s)= y^1(s)$ for $s\leq t$ and $y(s)=y^2(s-t)$ for $s\geq t$ is
$2\delta$-optimal from $x$ to $\K$, so we have  $\Gamma^\delta_x \subseteq \Gamma^{2\delta}_y \subseteq \Oc$, so $x \in \Oc_\delta$ by definition. 

Then, it is enough to show that for every $x \in \Oc_\delta$, $\dd_{\Oc}(x) \geq \frac{\lowerboundf}{2} \delta$. Suppose that there exists a $x \in \Oc_\delta$ such that $\dd_{\Oc}(x)< \frac{\lowerboundf}{2} \delta$. Then, we can construct a feasible trajectory starting from $x$, which first follows a straight line towards the closest point of the boundary, then goes back to $x$, and then follows an $\epsilon-$optimal trajectory to $\K$, with $\epsilon < \left(\delta -\frac{2 \dd_{\Oc}(x)}{\lowerboundf} \right)$. So this trajectory is $\delta-$optimal, and is not included in $\Oc$. We conclude the result by contradiction.
\end{proof}

\begin{lemma}\label{lemma_sc}
For every $x, y \in X^h \cap \Oc_\delta$, there exists a sequence of controls $( \bar{\alpha}^{x,y}_1 , \dots, \bar{\alpha}^{x,y}_N)$ such that, for the process $\{\xi_k \}$, we have $\xi_0 = x, \xi_N = y$, 
and 
\begin{equation}  
\E \left[ \sum_{k=1}^N \| \xi_k - \xi_{k-1} \| \right] = Nh \leq \sqrt{d} \|x - y \| \ .
\end{equation}
We denote $\bar{\alpha}^{x,y}$ this sequence of controls. 
\end{lemma}
\begin{proof}
One can construct such a sequence of controls by following the grid lines, that is taking $\lambda(x, \alpha, y_k) = 1$ for one particular $y_k$, and  $\lambda(x, \alpha, y_{k'}) = 0$ for other $y_{k'} \in Y^h(x+ \alpha h)$. In that case, $\{\xi_k\}$ is pure deterministic. 
\end{proof}	

For the full discretization scheme, we have the following convergence rate result.
\begin{theorem}\label{th_error_constrained}
Suppose~\Cref{assp1},~\Cref{assump_2},~\Cref{assump_3} and~\Cref{assump_5} hold.
Assume that $\Oc_{2 \delta}$ is nonempty and contains $\K$.
There exists a constants $C_{\Oc,1}', C_{\Oc,2}'$ depending on  $L_f, L_v, M_\Oc, M_f, M_t, M_f', M_t', \upperboundf, \lowerboundf$ such that, for every $0 < h < \min\{\frac{\delta}{C_1}, \frac{1}{\upperboundf} \}$:
\begin{equation}\label{zho-vh}
\sup_{x \in X^h \cap \Oc_{2\delta}} |z_\Oc^h(x) - v_\Oc(x) | \leq (C_{\Oc,1}'+\frac{C_{\Oc,2}'}{\delta}) h \ . 
\end{equation}
\end{theorem}
\begin{proof}
Take $x \in X^h \cap \Oc_{2\delta}$, it is enough to show that $|z_\Oc^h(x) - v^h(x) | \leq Ch$ for some constant $C$, as we already show $v^h_\Oc(x) = v^h(x)$ and $|v^h_\Oc(x) - v_\Oc(x) | \leq C_1 h$, for every $x \in \Oc_{\delta}$, in \Cref{propo-constraint}, and $\Oc_{2\delta } \subset \Oc_{\delta}$. 

Since $ \Oc_{\delta} \subset \R^d$, we have $z^h_\Oc(x) \geq z^h(x)$. Thus, $v^h(x) - z^h_\Oc(x) \leq v^h(x) - z^h(x) \leq C_{z_2} h $, where $C_{z_2}$ is the constant obtained in~\Cref{rate_v-wh}. 

For the reverse direction, 
let us denote $\alpha^*_x = ( \alpha^*_{x,0}, \alpha^*_{x,1}, \dots , \alpha^*_{x,N_{x}} )$  the sequence of optimal controls for the discrete time optimal control problem starting from $x$, and $y^h_k: = y^h_{\alpha^*_x}(x;k)$ the discrete optimal trajectory at time step $k \in \{1,2,\dots, N_x \} $.

Consider as in~\Cref{sec-fulldisfm}, the stochastic optimal control problem
with transition probabilities defined by the coefficients of the interpolation operation $I^x$, as in~\eqref{transition}, with the state constraint $\Oc$ and
target set $\K$.
This means that the process stops as soon as it hits either $\K$ or the boundary of $\Oc$ and that the cost in $\K$ is $0$ and in $\partial \Oc$ is $1$.

For any strategy $\sigma^h$, we denote $\bar{\alpha}^h = (\bar{\alpha}^h_1, \bar{\alpha}^h_2 , \dots )$ the random sequence of controls associated to this strategy, and by $\{ \xi_k\}$ the associated process. 
Moreover, denote $\bar{N}_\K = \min \{n \in \N_+ \mid \xi_n \in \K \}$ and $\bar{N}^{\Oc} = \min \{ n \in \N_+ \mid \xi_n \in  \Oc^{c} \}$. 
Then, $	z^h_\Oc$ satisfies
\[
z^h_\Oc(x) \leq \E^{\sigma^h}_x \left[ \left( 1 - \prod_{k=0}^{\bar{N}_\K} \Big(1 - \frac{h}{f(\xi_k, \bar{\alpha}^h_{k})} \Big)  \right) {\mathbbm 1}_{\bar{N}_\K<\bar{N}^{\Oc}}
+   {\mathbbm 1}_{\bar{N}_\K>\bar{N}^{\Oc}}\right]\enspace .
\]
Since the trajectory $y^h_k$ is deterministic, we have:
\begin{equation}\label{zho}
\begin{aligned}
	z^h_\Oc(x) - v^h(x)\leq \E^{\sigma^h}_x &\left[  \left( 1 - \prod_{k=0}^{\bar{N}_\K} \Big(1 - \frac{h}{f(\xi_k, \bar{\alpha}^h_{k})} \Big)  \right) {\mathbbm 1}_{\bar{N}_\K \leq \bar{N}^{\Oc}}
	+   {\mathbbm 1}_{\bar{N}_\K>\bar{N}^{\Oc}} \right. \\
	&\left. \quad  -\left(1 - \prod_{k=0}^{N_x} \Big(1 - \frac{h}{f(y^h_k, \alpha^*_{x,k})} \Big) \right)   \right] \enspace .
\end{aligned}
\end{equation}

We shall consider the strategy $\sigma^h$ for the stochastic optimal control problem with initial condition $\xi_0 = x$ defined as follows: 
to any history $H_k=(\xi_0, \alpha_0, \dots, \xi_{k-1}, \alpha_{k-1}, \xi_k)$, such that $\xi_l\in \Oc\setminus\K$ for $0\leq l\leq k$,
and  $k \leq N_x$, we set
$\sigma^h(H_k) = \alpha^*_{k,x}$. Moreover, if $\xi_{k} \in \Oc\setminus \K$ for all
$k\leq N_x$, then, we continue after time $N_x$ by taking a sequence of controls $\bar{\alpha}^{\xi_{N_x}, \xi_\K}$ as defined in~\Cref{lemma_sc}, with $\xi_\K$ be the closest point of $\xi_{N_x}$ in $\K \cap X^h$. 

To bound $z^h_\Oc(x) - v^h(x)$, we then consider the following
parts of the expectation in which
$\bar{N}_\K \leq \bar{N}^\Oc$;
$N_x < \bar{N}^\Oc < \bar{N}_\K$ and 
$\bar{N}^\Oc \leq \min \{N_x, \bar{N}_\K \}$.

\textbf{The part when $\bar{N}_\K \leq \bar{N}^\Oc$.}
In this case, the state constraint $\Oc$ is indeed never activated. We then follow a similar idea as in the proof of the unconstrained case, in particular the proof of~\Cref{rate_v-wh}. Notice that here we intend to bound the other direction $z^h_\Oc(x) - v^h(x)$, so we shall need the semiconcave condition of the speed function and the distance function to the target set $\K$, instead of the semiconvex condition used in the proof of~\Cref{rate_v-wh}. 
Let us consider first the part in which $N_x< \bar{N}_\K \leq\bar{N}^\Oc$. 

We omit similar computation steps here as in the proof of~\cref{rate_v-wh}. In particular, we have
\begin{equation}\label{zho-vh_1}
\begin{aligned}
	& \E^{\sigma^h}_x \left[ \left\{\left( 1 - \prod_{k=0}^{\bar{N}_\K} \Big(1 - \frac{h}{f(\xi_k, \bar{\alpha}^h_{k})} \Big)  \right) -\left(1 - \prod_{k=0}^{N_x} \Big(1 - \frac{h}{f(y^h_k, \alpha^*_{x,k})} \Big) \right) \right\} {\mathbbm 1}_{N_x<\bar{N}_\K \leq \bar{N}^\Oc}
	\right] \\ 
	& \leq \E^{\sigma^h}_x \left[ \left( \prod_{k=0}^{N_x} \Big(1 - \frac{h}{f(y^h_k, \alpha^*_{x,k})} \Big)  - \prod_{k=0}^{N_x} \Big(1 - \frac{h}{f(\xi_k, \alpha^*_{x,k})} \Big)  \right) {\mathbbm 1}_{N_x<\bar{N}_\K \leq \bar{N}^\Oc} \right. \\ 
	& \left. \quad + \left(  \prod_{k=0}^{N_x} \Big(1 - \frac{h}{f(\xi_k, \alpha^*_{x,k})} \Big)  \right) \left( 1 - \prod_{k = N_x +1}^{\bar{N}_\K } \Big(1 - \frac{h}{f(\xi_k, \bar{\alpha}^h_{k})} \Big) \right) 
	{\mathbbm 1}_{N_x<\bar{N}_\K \leq \bar{N}^\Oc}\right] \ .
\end{aligned}
\end{equation}
Following similar computations as in equation~\eqref{bound_dif} and equation~\eqref{bound_dif2}, using~\Cref{jensen_semiconcave} instead of ~\Cref{jensen_semiconvex}, we deduce that
\begin{equation}\label{bound_dif_vho1}
\E^{\sigma^h}_x \left[ \frac{1}{f(\xi_k, \alpha^*_{x,k})} - \frac{1}{f(y^h_k,\alpha^*_{x,k})} \mid k \leq N_x<\bar{N}_\K \right] \leq k M_f h^2 \ .
\end{equation}
Then, following similar computations as in equation~\eqref{vh-wh-1}, we deduce that 
\begin{equation}\label{bound_dif_vho11}
\E^{\sigma^h}_x \left[ \prod_{k=0}^{N_x} \Big(1 - \frac{h}{f(y^h_k, \alpha^*_{x,k})} \Big)  - \prod_{k=0}^{N_x} \Big(1 - \frac{h}{f(\xi_k, \alpha^*_{x,k})} \Big)  \mid N_x<\bar{N}_\K \right] \leq 2 M_f \upperboundf^2 h \ . 
\end{equation}
For the remaining part in~\eqref{zho-vh_1}, by a similar argument as for~\eqref{bound_dif_vho1}, and \eqref{bound_did},
we have 
\begin{equation}\label{distance_xi}
\E^{\sigma^h}_x  \left[ d_\K(\xi_k) - d_\K(y^h_k) \mid k\leq N_x <\bar{N}_\K \right] \leq k M_t h^2 \ .
\end{equation}
Then, by definition of the process $\{\xi_k\}$ after time $N_x$, constructed using 
\Cref{lemma_sc}, we obtain following the same computations as in~\eqref{vh-wh-2} 
\begin{equation}\label{bound_dif_vho2}
\E^{\sigma^h}_x  \left[  \left(  \prod_{k=0}^{N_x} \Big(1 - \frac{h}{f(\xi_k, \alpha^*_{x,k})} \Big)  \right) \left( 1 - \prod_{k = N_x +1}^{\bar{N}_\K } \Big(1 - \frac{h}{f(\xi_k, \bar{\alpha}^h_{k})} \Big) \right)  \mid N_x<\bar{N}_\K  \right] \leq \frac{\sqrt{d} M_t}{\lowerboundf} h \ .
\end{equation}

Now we can check that the part in which 
$\bar{N}_\K \leq\bar{N}^\Oc$ and $\bar{N}_\K \leq  N_x$ 
is bounded similarly to the part in which $\bar{N}_\K = N_x$,
and that this part is similar to the first term in the right hand side of
\eqref{zho-vh_1} (the only difference is the indicator function).

\textbf{The part when $N_x<\bar{N}^\Oc< \bar{N}_\K$.}
We have for all $k\in [N_x,\bar{N}_\K]$:
\begin{equation}\label{upp_xik}
\|\xi_k-\xi_\K\|_2\leq \|\xi_k-\xi_\K\|_1\leq \|\xi_{N_x}-\xi_\K\|_1\leq \sqrt{d} \|\xi_{N_x}-\xi_\K\|_2   \ .
\end{equation}
In particular, when $k=\bar{N}^\Oc$, we have 
\begin{equation}\label{lower_xik}
\|\xi_k-\xi_\K\|_2\geq d_\Oc(\xi_\K)\geq \frac{\lowerboundf}{2} \delta \ , 
\end{equation}
where the last inequality is deduced by a similar argument as in~\Cref{lemma_distancetobound} as $\K \subset \Oc_{2 \delta}$. 

Combining~\eqref{upp_xik} and~\eqref{lower_xik}, we have 
\begin{equation}
	1 \leq \frac{2\sqrt{d}}{\lowerboundf\delta}  \|\xi_{N_x} - \xi_{\K} \|_2 
	= \frac{2\sqrt{d}}{\lowerboundf\delta}  d_\K(\xi_{N_x})
	= \frac{2\sqrt{d}}{\lowerboundf\delta}  (d_\K(\xi_{N_x})-d_{\K}(y^h_{N_x})) \ .
\end{equation}
Thus, we have 
\begin{equation}\label{bound_part2}
	\begin{aligned}
		&\E^{\sigma^h}_x \left[ \left\{ 1 -  \left(1 - \prod_{k=0}^{N_x} \Big(1 - \frac{h}{f(y^h_k, \alpha^*_{x,k})} \Big) \right) \right\}{\mathbbm 1}_{N_x<\bar{N}_\Oc< \bar{N}_\K} \right] \\ 
		& \leq \E^{\sigma^h}_x  \left[ \left( \prod_{k=0}^{N_x} \Big(1 - \frac{h}{f(y^h_k, \alpha^*_{x,k})} \Big) \right) \frac{2\sqrt{d}}{\lowerboundf\delta}
		(d_\K(\xi_{N_x})-d_{\K}(y^h_{N_x}))\right] \leq \frac{2\sqrt{d} M_t \upperboundf}{\lowerboundf\delta} h \ ,
	\end{aligned}
\end{equation}
where the last inequality is deduced from~\eqref{distance_xi} by taking $k = N_x$. 

\textbf{The part when $\bar{N}_\Oc \leq \min \{N_x, \bar{N}_\K \}$.}
In this case, $\{\xi_k\}$ will first touch the boundary of $\Oc$ before arriving in the target set $\K$ and before $y^h_k$ is arriving in $\K$.


By~\Cref{jensen_semiconvex}, under~\Cref{assump_5}, we have 
\begin{equation}\label{upper_dis_xiN}
	\E^{\sigma^h}_x \left[ \tilde{d}_{\Oc}(y^h_k) - \tilde{d}_\Oc(\xi_k) \mid k \leq \bar{N}_{\Oc}\leq \min \{N_x, \bar{N}_\K \} \right] \leq k M_\Oc h^2 \ .
\end{equation}
In particular, when $k = \bar{N}_\Oc$, by~\Cref{lemma_distancetobound} we have 
\begin{equation}\label{lower_dis_xiN}
	\tilde{d}_{\Oc}(y^h_{\bar{N}_\Oc}) - \tilde{d}_\Oc(\xi_{\bar{N}_\Oc}) \geq \frac{\lowerboundf}{2} \delta.  
\end{equation}
Combining~\Cref{upper_dis_xiN} and~\Cref{lower_dis_xiN}, we have 
\begin{equation}
	1 \leq \frac{2 \bar{N}_\Oc M_\Oc h^2}{\lowerboundf \delta} \ .
\end{equation}
Thus, we have 
\begin{equation}\label{bound_part3}
	\E^{\sigma^h}_x \left[ \left\{ 1 -  \left(1 - \prod_{k=0}^{N_x} \Big(1 - \frac{h}{f(y^h_k, \alpha^*_{x,k})} \Big) \right) \right\}{\mathbbm 1}_{\bar{N}_\Oc \leq \min \{N_x, \bar{N}_\K \}} \right] \leq \frac{2 M_\Oc \upperboundf}{\lowerboundf \delta } h \ .
\end{equation}

Combining the bounds on the three parts of expectation, that is~\eqref{bound_dif_vho11}, \eqref{bound_dif_vho2}, \eqref{bound_part2} and~\eqref{bound_part3}, we have
\begin{equation}\label{bound_zho-v}
	z^h_\Oc(x) - v^h(x) \leq ( 2M_f \upperboundf^2 + \frac{\sqrt{d} M_t}{\lowerboundf} + \frac{ 2\sqrt{d} M_t \upperboundf  + 2 M_\Oc \upperboundf} {\lowerboundf \delta} ) h \ .
\end{equation}

\end{proof}

\begin{remark}
By~\Cref{th_error_constrained}, when $\delta$ is large, that is when a point $x$ has an optimal trajectory enough far away from the boundary, the error bound of the full discretization scheme of a state constrained problem is still in the order of $h$. Otherwise, one get an error bound in the order of $\frac{h}{\delta}$. 
\end{remark}

\begin{remark}
In order to obtain an error bound of $O(h)$ of the fast marching method, in some set of points $\mathcal{X}$, one shall need to generate a grid that covers at least $\delta-$optimal geodesic from all points of $\mathcal{X}$ to $\K$, with $\delta \gg h$. 
\end{remark}

\begin{remark}In the recent work~\cite{akian2023multi}, the authors introduced a multilevel fast-marching method. 
There, they show that the computational complexity of this method is a function of the convergence rate of the original fast-marching method (or of the discretization), and of the "stiffness" of the value function. 
A condition on the convergence rate is as follows.
Let $\Oc'_{\eta}$ be the $\eta$-geodesic set from one set 
$\sourceset$ to another $\destset=\K$ as in~\cite{akian2023multi},
for the problem with no constraints.
Then, the assumption (A3) in \cite{akian2023multi} means that
the error of discretization is in the order of $h^\gamma$, and that
the same error bound holds for the points in $\Oc'_{\eta/2}$ if we consider the
problem with state constraint in state space $\Oc'_{\eta}$. Moreover this
assumption is applied to $\eta$ in the order of $h^\gamma$.

If $\Oc= \Oc'_{\eta}$, we obtain $\Oc'_{\eta/2}\subset \Oc_{\eta/2}$.
Therefore the bound in \Cref{th_error_constrained} applies for the points in
$\Oc'_{\eta/2}$, when $\delta=\eta/4$.
The difficulty is that in \Cref{th_error_constrained}, the convergence rate is 
depending on $\delta$.
Therefore, one can only deduce the assumptions of \cite{akian2023multi}
with $\gamma=1/2$, which in turn do not lead to any improvement of the
complexity with respect to the usual fastmarching method, since we
already proved that the error is in the order of $h$.
To obtain an improvement, we need to show the bound in \Cref{th_error_constrained} with a constant independent of $\delta$.
\end{remark}

\bibliographystyle{alpha} 

\newcommand{\etalchar}[1]{$^{#1}$}

\end{document}